\documentclass[12 pt]{amsart}

\usepackage{enumerate,xypic,mathdots,mathrsfs,verbatim,amssymb,url,graphicx}
\usepackage{fullpage}
\usepackage{comment,subfig}
\usepackage[table]{xcolor}
\usepackage{multicol}
\usepackage{supertabular}
\usepackage{longtable}
\usepackage{tikz}
\usepackage{tabularx}
\usepackage{float}

\usepackage{lmodern}

\makeatletter
\ifcase \@ptsize \relax
  \newcommand{\miniscule}{\@setfontsize\miniscule{4}{5}}
\or
  \newcommand{\miniscule}{\@setfontsize\miniscule{5}{6}}
\or
  \newcommand{\miniscule}{\@setfontsize\miniscule{5}{6.7}}
\fi
\makeatother

\makeatletter
\let\mcnewpage=\newpage
\newcommand{\TrickSupertabularIntoMulticols}{%
  \renewcommand\newpage{%
    \if@firstcolumn
      \hrule width\linewidth height0pt
      \columnbreak
    \else
      \mcnewpage
    \fi
  }%
}
\makeatother

\numberwithin{equation}{section}

\newtheorem{prop}{Proposition}[section]
\newtheorem{theorem}[prop]{Theorem}
\newtheorem{cor}[prop]{Corollary}

\theoremstyle{definition}
\newtheorem{defn}[prop]{Definition}
\newtheorem{example}[prop]{Example}

\theoremstyle{remark}
\newtheorem{rem}[prop]{Remark}

\newcommand{\A}{\mathscr{A}}

\newcommand{\C}{\mathbb{C}}

\newcommand{\F}{\mathbb{F}}
\newcommand{\G}{\mathcal{G}}

\renewcommand{\H}{\mathcal{H}}
\newcommand{\HS}{\mathcal{HS}}

\newcommand{\R}{\mathbb{R}}

\newcommand{\T}{\mathbb{T}}

\newcommand{\Z}{\mathbb{Z}}

\newcommand{\Norm}[1]{\left\Vert #1 \right\Vert}

\renewcommand{\subset}{\subseteq}
\renewcommand{\supset}{\supseteq}

\DeclareMathOperator{\Aut}{Aut}

\DeclareMathOperator{\Inn}{Inn}
\DeclareMathOperator{\Mat}{Mat}
\DeclareMathOperator{\mult}{mult}

\DeclareMathOperator{\spn}{span}

\DeclareMathOperator{\SL}{SL}
\DeclareMathOperator{\Sp}{Sp}

\DeclareMathOperator{\tr}{tr}


\begin{document}

\title{Optimal line packings from finite group actions}

\author{Joseph W.\ Iverson}
\address{Department of Mathematics, University of Maryland, College Park, MD 20742}
\email{jiverson@math.umd.edu}

\author{John Jasper}
\address{Department of Mathematics and Statistics, South Dakota State University, Brookings, SD 57007}
\email{john.jasper@sdstate.edu}

\author{Dustin G.\ Mixon}
\address{Department of Mathematics, The Ohio State University, Columbus, OH 43210}
\email{mixon.23@osu.edu}

\date{\today}
\maketitle

\begin{abstract}
We provide a general program for finding nice arrangements of points in real or complex projective space from transitive actions of finite groups.
In many cases, these arrangements are optimal in the sense of maximizing the minimum distance.
We introduce our program in terms of general Schurian association schemes before focusing on the special case of Gelfand pairs.
Notably, our program unifies a variety of existing packings with heretofore disparate constructions.
In addition, we leverage our program to construct the first known infinite family of equiangular lines with Heisenberg symmetry.
\end{abstract}

\section{Introduction} \label{sec:intro}

We consider the fundamental problem of packing points in real or complex projective space so that the minimum distance is maximized.
A famous instance of this problem is the \emph{Tammes problem}~\cite{Tammes:30}, which concerns the packing of points in $\mathbb{S}^2\cong\mathbb{C}\mathbf{P}^1$.
In this space, the optimal packing of 13 points was the subject of a celebrated argument between Newton and Gregory~\cite{Newton:66}.
Recently, the general problem of packing in projective space has received renewed attention due to its applications in communication, coding, and quantum information theory~\cite{SH,Z}.
In fact, the last few years produced a multitude of disparate constructions of optimal packings~\cite{fickus2012steiner,bodmann2016achieving,fickus2016equiangular,fickus2018tremain,fickus2016polyphase,casazza2016geometric,IJM,fickus2017hadamard,appleby2017constructing,FJM17} (see~\cite{FMTable} for a living survey), leaving one yearning for some sort of unified theory.
This paper provides a modest step in that direction by identifying a fruitful correspondence with transitive actions of finite groups.

Points in projective space correspond to one-dimensional subspaces (lines) of some real or complex vector space, and for convenience, we represent each line with a spanning unit vector.
Our packing problem then amounts to finding unit vectors $\{\phi_i\}_{i=1}^n$ that minimize \emph{coherence}, defined by
\[
\max_{1 \leq i \neq j \leq n} |\langle \phi_i,\phi_j\rangle|.
\]
To minimize coherence, it suffices to achieve equality in some known lower bound, such as the Welch, orthoplex or Levenstein bounds~\cite{Wel74,ConwayHS:96,levenshtein1982bounds}.
For each of these bounds, there exist specific cases in which equality is achievable.
Interestingly, a packing achieves equality in the Welch bound precisely when the Gram matrix $\begin{bmatrix} \langle \phi_j,\phi_i\rangle \end{bmatrix}_{i,j=1}^n$ is a scalar multiple of a projection with off-diagonal entries of constant modulus~\cite{SH}.
Such packings are known as \emph{equiangular tight frames (ETFs)}.

Conway, Hardin and Sloane~\cite{ConwayHS:96} were perhaps the first to observe that highly symmetric arrangements of lines are frequently strong competitors in the packing problem.
Packings exhibiting abelian symmetry are known as \emph{harmonic frames}, and harmonic ETFs are constructed using combinatorial objects known as difference sets~\cite{SH,XZG}.
Optimal packings of $d^2$ points in $\mathbb{C}\mathbf{P}^{d-1}$ are conjectured to be ETFs with Heisenberg symmetry for every $d$, which correspond to desirable measurement ensembles for quantum state estimation~\cite{Z}.
As a precursor to the present paper, the authors recently used group schemes to construct the first known infinite family of ETFs with nonabelian symmetry~\cite{IJM}.

The main idea of this paper is illustrated in Figure~\ref{fig:big_picture}.
Every transitive action of a finite group determines a \emph{Schurian association scheme}, which in turn produces a collection of distinguished projections through its \emph{spherical functions}. Each projection inherits symmetries from the group action, and has a (small) number of distinct entries bounded by the dimension of the scheme's adjacency algebra.
Viewing each projection as a Gram matrix then produces a collection of vectors that will often generate an optimal line packing.
In particular, the packing might be an ETF since it necessarily has a small number of angles.

While each of these individual relationships is known, the entire chain suggests a useful new discovery tool for researchers.
For example, one may systematically search through finite group actions in GAP~\cite{gap} to find worthy line packings.
The authors used this program to find an ETF exhibiting Heisenberg symmetry, and then generalized it to the first known infinite family with such symmetry.
While these ETFs are not exactly the packings desired in quantum information theory, we expect there to be some sort of relationship (as in~\cite{appleby2017dimension}), and we leave this for future investigation.

The following section covers preliminary information about Schurian association schemes and in particular, the commutative instances corresponding to \emph{Gelfand pairs}. 
In the theory of Lie groups, Gelfand pairs are used widely for the reproducing properties of their spherical functions on homogeneous spaces \cite{Hel00}.
However, Gelfand pairs of finite groups appear to have received comparatively little attention from the frame theory community.
As far as we are aware, the current article represents the first systematic attempt to mine Gelfand pairs as sources of finite frames.
Section~3 then discusses the packings that arise from Schurian schemes, known as \emph{homogeneous frames}. We illustrate the theory with examples in Section~4.
Sections~5 and~6 then explain how to leverage the chain of relationships illustrated in Figure~\ref{fig:big_picture} to produce an infinite family of ETFs with Heisenberg symmetry. A nontrivial consequence of Zauner's conjecture \cite{Z} is that an infinite family like this exists, and so our construction gives theoretical evidence in favor of that conjecture.

\begin{figure}[t]
\begin{center}
\begin{tikzpicture}
\draw[thick,->] (2.5,-1) -- (4,-1);
\draw[thick,<-] (0,-0.3) -- (0,1);
\draw[thick,->] (6.5,0.3) -- (6.5,1);
\node[align=center] at (0,3.3) {\emph{Discrete world}};
\node[align=center] at (6.5,3.3) {\emph{Continuous world}};
\node[draw,text width=3.5cm,align=center] at (0,2) {transitive actions\\ of finite groups};
\node[draw,text width=3.5cm,align=center] at (0,-1) {Schurian schemes};
\node[draw,text width=3.5cm,align=center] at (6.5,2) {optimal line\\ packings};
\node[draw,text width=3.5cm,align=center] at (6.5,-1) {projections with\\ symmetries and\\ few distinct entries};
\end{tikzpicture}
\end{center}
\caption{
Important chain of relationships in this paper.
}  \label{fig:big_picture}
\end{figure}
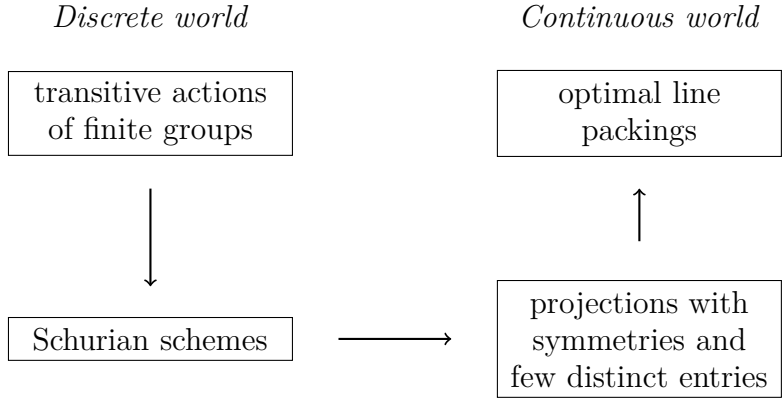

\section{Preliminaries} \label{sec:pre}

We begin by recalling the basic theory of frames, Schurian association schemes, and Gelfand pairs. See \cite{christensen2003introduction,CKP,BI,CSST} for more details.

\subsection{Frames}
Let $\H$ be a $d$-dimensional Hilbert space ($d < \infty$), either real or complex. A sequence of vectors $\Phi = \{ \phi_j \}_{j=1}^n$ in $\H$ is called a \emph{frame} if there are constants $0 < A \leq B < \infty$ such that
\[ A \Norm{ \psi }^2 \leq \sum_{j=1}^n | \langle \psi, \phi_j \rangle |^2 \leq B \Norm{ \psi }^2 \]
for all $\psi \in \H$. We call $A$ and $B$ the \emph{frame bounds}. When $A=B$, the frame is called \emph{tight}, and when $A = B = 1$, it is called \emph{Parseval}. In abuse of notation, we sometimes think of $\Phi$ as a short, fat matrix whose columns describe the frame vectors $\phi_1,\dotsc,\phi_n$. With this in mind, the \emph{Gram matrix} for $\Phi$ is the $n\times n$ matrix
\[ \Phi^* \Phi = \begin{bmatrix} \langle \phi_j, \phi_i \rangle \end{bmatrix}_{i,j=1}^n, \]
which is clearly positive semidefinite. To say that $\Phi$ is a tight frame for its span means precisely that $\Phi^* \Phi$ is a constant times a projection. Moreover, it is possible to recover $\Phi$ from its Gram matrix, up to a unitary equivalence. (For example, when $\Phi$ is Parseval we can just take the columns of $\Phi^* \Phi$ in their span.)
We say that Parseval frames $\Phi$ and $\Psi$ are \emph{Naimark complements} if $\Phi^*\Phi+\Psi^*\Psi=I$, meaning their Gram matrices project onto orthogonally complementary spaces.

In this paper, we want to find \emph{equiangular tight frames} (ETFs), which are tight frames with two more properties: (1) All of the frame vectors have the same (nonzero) length, and (2) The inner product $| \langle \phi_i, \phi_j \rangle |$ is constant across all pairs of distinct frame vectors $\phi_i \neq \phi_j$. When an ETF occurs, it can be rescaled so that its vectors are unit norm with coherence matching the Welch bound
\[ \mu(\Phi):= \max_{1 \leq i \neq j \leq n} | \langle \phi_i, \phi_j \rangle | \geq \sqrt{ \frac{n-d}{d(n-1)} }. \]
In particular, an ETF is an optimal line packing.

Our strategy is to hunt for ETFs via their Gram matrices, which are recognizable from three features: (1) They are constant multiples of projections, (2) They have constant diagonal, and (3) They have constant modulus off the diagonal. We are going to look in particular for Gram matrices that lie in the adjacency algebras of \emph{Schurian association schemes}, described below.

\subsection{Schurian schemes}
Let $G$ be a finite group acting transitively on a finite set $X$, from the left. This action determines a matrix algebra in the following way. Let $G$ act on $X\times X$ by setting $g\cdot (x,y) = (g\cdot x, g\cdot y)$ for $g\in G$ and $x,y\in X$, and let $R_0,\dotsc,R_c \subset X \times X$ be the orbits of this action, indexed in such a way that $R_0 = \{(x,x) : x \in X\}$. We can express each orbit $R_i$ as an $X\times X$ matrix $A_i$ whose entries are given by
\begin{equation} \label{eq:Ai}
(A_i)_{x,y} = \begin{cases}
1, & \text{if }(x,y) \in R_i; \\
0, & \text{otherwise.}
\end{cases}
\end{equation}
Then $A_0 + \dotsb + A_c = J$, the matrix of all ones, and it turns out that $\A := \spn\{A_0,\dotsc,A_c\}$ is a complex $*$-algebra under the usual matrix multiplication. In other words, $A_0,\dotsc,A_c$ form a (possibly noncommutative) \emph{association scheme}. Schemes of this type are called \emph{Schurian}, or \emph{group-case}. We call $A_0,\dotsc,A_c$ the \emph{adjacency matrices} of the scheme, and $\A$ the \emph{adjacency algebra}; it consists of all complex $X\times X$ matrices $M$ with the property that $M_{g\cdot x, g\cdot y} = M_{x,y}$ for all $g\in G$ and $x,y \in X$. In other words, $\A$ is the $*$-algebra of all \emph{$G$-stable} matrices.

For Schurian schemes, the adjacency algebra has another description in terms of representation theory. Let $L^2(X)$ be the space of complex-valued functions on $X$, with the inner product
\[ \langle f, g \rangle_{L^2(X)} = \sum_{x\in X} f(x) \overline{g(x)} \qquad \left(f,g \in L^2(X)\right). \]
Using the canonical basis of point masses in $L^2(X)$, we can think of our adjacency matrices as linear operators $A_0,\dotsc,A_c \in B(L^2(X))$. The action of $G$ on $X$ produces a unitary representation $\lambda \colon G \to U(L^2(X))$,
\[ [\lambda(g)f](x) = f(g^{-1}\cdot x) \qquad \left(g \in G,\, f \in L^2(X),\, x \in X\right), \]
and the adjacency algebra $\A$ coincides with its commutant
\[ \mathcal{C}(\lambda) = \{ T \in B(L^2(X)) : T\lambda(g) = \lambda(g)T \text{ for all }g\in G\}. \]

For a third realization of the adjacency algebra, fix a point $x_0 \in X$ and let $H = G_{x_0}$ be its stabilizer in $G$. As a $G$-set, $X$ is isomorphic to $G/H$ with the usual action on the left. For each $i=0,\dotsc,c$, there is a unique double coset
\[ H a_i H = \{ h a_i h' : h,h' \in H\} \]
for $H$ in $G$ such that
\begin{equation} \label{eq:Ai2}
(A_i)_{g\cdot x_0, h\cdot x_0} = \begin{cases}
1, & \text{if }h^{-1}g \in Ha_i H; \\
0, & \text{otherwise.}
\end{cases}
\end{equation}
The adjacency matrices are in one-to-one correspondence with the double cosets through the mapping $A_i \mapsto Ha_iH$. Moreover, for any $A\in \A$, the \emph{convolution kernel} $\psi_A \colon G \to \C$ given by
\begin{equation} \label{eq:psiA}
\psi_A(g) = \frac{1}{|H|} A_{g\cdot x_0, x_0} \qquad (g\in G)
\end{equation}
belongs to the space
\[  L^2(H\backslash G / H) = \{ \psi \colon G \to \C : \psi(gh) = \psi(hg) = \psi(g) \text{ for all $g\in G, h \in H$} \} \]
of bi-$H$-invariant functions on $G$. This has the structure of a $*$-algebra with the usual convolution and involution,
\[ (\psi_1 * \psi_2)(g) = \sum_{h\in G} \psi_1(h) \psi_2(h^{-1} g) \qquad \text{and} \qquad \psi^*(g) = \overline{\psi(g^{-1})} \qquad (g \in G), \]
and if we set $\tilde{\psi}_A(g) = \psi_A(g^{-1})$, the mapping $A \mapsto \tilde{\psi}_A$ is a $*$-algebra isomorphism from $\A$ to $L^2(H\backslash G / H)$.

\begin{defn}
We call $(G,H)$ a \emph{Gelfand pair} when $\A = \mathcal{C}(\lambda) \cong L^2(H\backslash G / H)$ is commutative.
\end{defn}

This notion, too, has a representation-theoretic interpretation. Let $\chi_\lambda \colon G \to \C$ be the trace character of $\lambda$, namely
\[ \chi_\lambda(g) = \left| \{x\in X : g\cdot x = x\}\right| \qquad (g \in G), \]
and let
\[ \chi_\lambda = n_0 \chi_0 + \dotsb + n_r \chi_r \]
be its decomposition into irreducible characters $\chi_0,\dotsc,\chi_r$, with $n_j \geq 1$ for all $j$, and $\chi_i \neq \chi_j$ for $i \neq j$. Then $\A$ is commutative if and only if $n_j = 1$ for all $j=0,\dotsc,r$. In other words, $(G,H)$ is a Gelfand pair if and only if $\lambda$ is \emph{multiplicity free}. 

\subsection{Spherical functions} \label{subsec:sph}

Whether or not $\A$ is commutative, many of its projections can be constructed explicitly from the character table of $G$, as follows. Each of the constituent characters $\chi_j$ above determines a \emph{spherical function} $\omega_j \in L^2(H\backslash G / H)$ given by
\begin{equation} \label{eq:omj}
\omega_j(g) = \frac{1}{|H|} \sum_{h\in H} \chi_j(g^{-1} h) = \frac{1}{|HgH|} \sum_{h\in HgH} \overline{\chi_j(h)} \qquad (g \in G).
\end{equation}
Writing $m_j = \chi_j(1_G)$ for the degree of $\chi_j$, the matrix
\begin{equation} \label{eq:Ej}
E_j := \frac{m_j}{|X|}\sum_{i=0}^c \omega_j(a_i) A_i
\end{equation}
is orthogonal projection onto the isotypical component $V_j \subset L^2(X)$ corresponding to $\chi_j$. In other words, $V_j$ is the unique $\lambda$-invariant subspace of $L^2(X)$ on which the restriction of $\lambda$ has trace character $n_j \chi_j$. Since $V_j$ is $\lambda$-invariant, $E_j$ is an orthogonal projection in $\mathcal{C}(\lambda) = \A$ with rank $n_j m_j$. Moreover, $E_0+\dotsb+E_r = I$, and $E_i E_j = E_j E_i = 0$ whenever $i \neq j$. In particular, every sum $\sum_{j\in D} E_j$ for $D\subset \{0,\dotsc,r\}$ is an orthogonal projection in $\A$. If it happens that $\A$ is commutative, then $E_0,\dotsc,E_c$ form the basis of primitive idempotents promised by the spectral theorem, and every projection in $\A$ takes the form just described. We summarize these results below.

\begin{prop} \label{prop:GD}
Every subset $D \subset \{0,\dotsc,r\}$ produces an orthogonal projection $\mathcal{G}_D = \sum_{j \in D} E_j$ in $\A$ with entries
\begin{equation}\label{eq:GD}
(\mathcal{G}_D)_{g\cdot x_0, h \cdot x_0} = \frac{1}{|X|}\sum_{j\in D} m_j \omega_j(h^{-1} g) \qquad (g,h \in G).
\end{equation}
When $(G,H)$ is a Gelfand pair, every orthogonal projection in $\A$ takes this form.
\end{prop}

We have written code for the computer program GAP \cite{gap} to compute the spherical functions associated with any transitive group action, and to produce the corresponding projections \cite{github2}.

The spherical functions have another description in terms of invariant vectors. If $\pi_j \colon G \to U(\H_j)$ is any unitary representation of $G$ affording $\chi_j$ as its trace character, then the space
\[ \H_j^H = \{ v \in \H_j : \pi_j(h)v = v \text{ for all $h\in H$} \} \]
of $H$-stable vectors in $\H_j$ has dimension $n_j$. If $u_1,\dotsc,u_{n_j}$ is an orthonormal basis for $\H_j^H$, then
\begin{equation} \label{eq:spd}
\omega_j(g) = \sum_{i=1}^{n_j} \langle u_i, \pi_j(g) u_i \rangle \qquad (g \in G).
\end{equation}

\subsection{Examples}

\begin{example} \label{ex:J}
$G$ acts trivially on the subspace of constant functions in $L^2(X)$, so one of the constituents of $\chi_\lambda$, say $\chi_0$, is the trivial character $\chi_0(g) \equiv 1$. The corresponding spherical function $\omega_0$ is constantly equal to $1$. Hence,
\[ E_0 = \frac{1}{|X|} \sum_{i=0}^c A_i = \frac{1}{|X|} J, \]
which is indeed orthogonal projection onto the subspace of constant functions.
\end{example}

\begin{example} \label{ex:hrm}
Let $G$ be any finite group, acting on $X=G$ by left translation. The adjacency algebra $\A$ consists of all \emph{$G$-circulant} matrices. Since the stabilizer of any point in $G$ is the trivial group $H=\{1_G\}$, $\A$ is isomorphic to the $*$-algebra $L^2(G)$. Thus, we have a Gelfand pair if and only if $G$ is abelian. In that case, the spherical functions are given by the Pontryagin dual group $\hat{G}$, which consists of all homomorphisms $\alpha \colon G \to \T$, under the operation of pointwise multiplication. Indeed, the permutation representation $\lambda \colon G \to L^2(X)$ is the left regular representation of $G$, and the Peter-Weyl Theorem tells us that every character $\alpha \in \hat{G}$ appears as a constituent of $\chi_\rho$. From \eqref{eq:omj}, we see that the spherical function corresponding to $\alpha \in \hat{G}$ is $\overline{\alpha} = \alpha^{-1}$. Any choice of subset $D \subset \hat{G}$ prescribes, via Proposition~\ref{prop:GD}, a $G\times G$ orthogonal projection $\mathcal{G}_D$ with entries
\begin{equation} \label{eq:hrm}
(\G_D)_{g,h} = \frac{1}{|G|} \sum_{\alpha \in D} \alpha(h) \overline{\alpha(g)} \qquad (g,h \in G).
\end{equation}
\end{example}

\begin{example} \label{ex:gs}
Let $K$ be any finite group (possibly nonabelian), and let $G = K \times K$, acting on $X=K$ by the formula $(g,h)\cdot k = gkh^{-1}$ for $g,h,k \in K$. The orbits of the corresponding action on $X\times X$ are indexed by the conjugacy classes $\mathcal{C}_0,\dotsc,\mathcal{C}_c$ of $K$, and take the form
\[ R_i = \{(gh,h) : g \in \mathcal{C}_i, h \in K\}. \]
If $L_g = \begin{bmatrix} \delta_{gh,h} \end{bmatrix}_{g,h \in K}$ is the $K\times K$ matrix representation for left translation by $g\in K$ on $L^2(K)$, it follows that
\[ A_i = \sum_{g\in \mathcal{C}_i} L_g. \]
The adjacency algebra $\A$ is the center of the group algebra of $K$-circulant matrices described in Example~\ref{ex:hrm}. The stabilizer of the point $1_K \in K$ is the diagonal 
\[ \Delta(K) = \{(g,g) : g \in K\}. \]
Since $\A$ is commutative, $(K\times K, \Delta(K))$ is a Gelfand pair.

As in the abelian case, primitive idempotents in $\A$ are indexed by the set $\hat{K}$ of irreducible characters of $K$, with a character $\chi \in \hat{K}$ corresponding to the matrix
\[ E_\chi := \frac{\chi(1)}{|K|} \begin{bmatrix} \chi(g^{-1} h) \end{bmatrix}_{g,h \in K}. \]
See, for instance, \cite[Thm.~10.6.1]{Godsil}. Once again, projections in $\A$ are uniquely determined by subsets $D \subset \hat{K}$, through the formula
\begin{equation} \label{eq:cgf}
(\G_D)_{g,h} = \frac{1}{|K|} \sum_{\chi \in D} \chi(1) \chi(g^{-1} h) \qquad (g,h \in G).
\end{equation}
\end{example}

This completes our review of background material.

\section{Homogeneous frames and Schurian schemes}

In general, we are interested in association schemes primarily as sources of finite frames, represented by their Gram matrices in the corresponding adjacency algebra.
Any positive semidefinite matrix can be viewed as the Gram matrix of some frame, and the resulting frame is Parseval if and only if the Gram matrix is a projection.
By the spectral theorem, any commutative $*$-algebra of square matrices therefore determines a finite set of Parseval frames through its projections.
In the case of association schemes, the resulting frames have few inner products---no more than the number of adjacency matrices.
In this sense, association schemes may be well suited for the construction of low-coherence tight frames.

In fact, the adjacency algebra of any association scheme of $n\times n$ matrices contains the Gram matrices of three trivial ETFs: an orthonormal basis for $\C^n$ (represented by the identity matrix $I$), $n$ identical vectors in $\C^1$ (represented by $\frac{1}{n} J$), and the $n$-vector \emph{simplex} in $\C^{n-1}$ (represented by $I - \frac{1}{n}J$).
Note that these last two examples are Naimark complements of each other, and in general, adjacency algebras are closed under such complementation.

Among association schemes, the Schurian schemes are particularly attractive for two reasons.
First, they provide a channel from the discrete world of finite group actions to the continuous setting of finite frames in $\C^d$.
Second, the spherical functions make it easy to compute projections from a character table. 
Here, the identity matrix corresponds to $D=\{0,\dotsc,r\}$ in Proposition~\ref{prop:GD}, whereas projection onto the all-ones vector comes from the trivial action of $G$ on constant functions in $L^2(X)$, as in Example~\ref{ex:J}.
In addition, the Naimark complement corresponds to the set complement of $D$.
Now that we have established how trivial ETFs naturally arise from Schurian schemes, we are ready to pursue nontrivial constructions.

This section is devoted to the general theory of frames whose Gram matrices have Schurian structure. We begin by relating symmetry in the Gram matrix to symmetry in the frame vectors themselves. Subsection~\ref{subsec:homCon} continues with a series of techniques to identify group structure in a given frame. In subsection~\ref{subsec:reg}, we explain how this machinery distinguishes a small class of frames associated with a \emph{regular subgroup}. Finally, subsection~\ref{subsec:proj} introduces an important technical tool for squeezing additional line packings out of a given scheme. Illustrative examples are sprinkled throughout this section; more substantial examples appear in Sections~4 and~6.

\subsection{Homogeneous frames} \label{subsec:hom}

\begin{defn}
Let $G$ be a finite group, and let $\rho \colon G \to U(\H)$ be a unitary representation. Any frame of the form $\Phi = \{ \rho(g)v \}_{g\in G}$, with $v\in \H$, is called a \emph{group frame}, or more specifically, a \emph{$G$-frame}. If there is a subgroup $H \subset G$ such that $\rho(h)v = v$ for all $h \in H$, then we can reduce $\Phi$ to form a new frame, $\Phi' = \{ \rho(g)v \}_{gH \in G/H}$. We call $\Phi'$ a \emph{homogeneous frame}, or if we wish to emphasize the particular groups involved, a \emph{$(G,H)$-frame}.
\end{defn}

\begin{theorem} \label{thm:hgf}
Let $G$ be a finite group acting transitively on a set $X$, and let $\A$ be the adjacency algebra of all $G$-stable, $X \times X$ matrices. After changing indices through any $G$-set isomorphism $X \cong G/H$ for a subgroup $H\subset G$, the positive semidefinite matrices in $\A$ are precisely the Gram matrices for $(G,H)$-frames.
\end{theorem}

\begin{proof}
First, let $\Phi' = \{ \rho(g) v \}_{gH \in G/H}$ be a $(G,H)$-frame. Then its Gram matrix $\G$ is positive semidefinite, with entries given by
\[ \G_{gH,hH} = \langle \rho(h) v, \rho(g) v \rangle \]
for $gH,hH \in G/H$. For any $k\in G$, we have
\[ \G_{kgH, khH} = \langle \rho(kh) v, \rho(kg) v \rangle = \langle \rho(k) \rho(y), \rho(k) \rho(g) v \rangle = \G_{gH,hH}, \]
since $\rho(k)$ is unitary. Hence, $\G \in \A$.

In the other direction, if $\G$ is any positive semidefinite matrix in $\A$, then its convolution kernel $\psi_\G = \tilde{\psi}_{\G^T}$ is a positive element of the finite-dimensional $*$-algebra $L^2(H \backslash G / H)$, so it has a unique positive square root $\psi_\G^{1/2} \in L^2(H\backslash G / H)$. For each $g \in G$, let $L_g \colon L^2(G) \to L^2(G)$ be the translation operator given by $(L_g \psi)(h) = \psi(g^{-1}h)$. We define $v = |H|^{1/2} \psi_\G^{1/2}$ and
\[ \H = \spn\{ L_g v : g \in G \} \subset L^2(G), \]
and let $\lambda\colon G \to U(\H)$ describe left translation on $\H$. As a finite spanning set, $\Phi:=\{ \lambda(g) v \}_{g\in G}$ is a frame for $\H$. Since $v \in L^2(H \backslash G / H)$ is $H$-stable, we obtain a $(G,H)$-frame $\Phi' := \{ \lambda(g) v \}_{gH \in G/H}$.

It remains to show that $\G$ is the Gram matrix for $\Phi'$. For any $g,h \in G$, we have
\[ \langle \lambda(h) v, \lambda(g) v \rangle = |H| \langle \psi_\G^{1/2}, \lambda(h^{-1} g) \psi_\G^{1/2} \rangle = |H| \sum_{k\in G} \psi_\G^{1/2}(k) \overline{ \psi_\G^{1/2}(g^{-1} h k) }, \]
and since $\psi_\G^{1/2}$ is self-adjoint,
\[ \langle \lambda(h) v, \lambda(g) v \rangle = |H| \sum_{k\in G}\psi_\G^{1/2}(k)  \psi_\G^{1/2}(k^{-1} h^{-1} g) = |H| (\psi_\G^{1/2} * \psi_\G^{1/2})(h^{-1} g) = |H| \psi_\G(h^{-1} g). \]
By \eqref{eq:psiA}, $\langle \lambda(h) v, \lambda(g) v \rangle = \G_{h^{-1}gH,H} = \G_{gH,hH}$, as desired.
\end{proof}

\begin{example} \label{ex:gphg}
We now explain how to produce homogeneous frames with Gram matrices as in Proposition~\ref{prop:GD}. Following the notation of Section~\ref{sec:pre}, fix a subset $D \subset \{0,\dotsc,r\}$, and let $\rho = \bigoplus_{j\in D} \pi_j^{(n_j)}$, where $\pi_j^{(n_j)}$ denotes the direct sum of $n_j = \mult(\pi_j, \lambda)$ copies of $\pi_j$. For each $j \in D$, choose an orthonormal basis $u_1^{(j)},\dotsc,u_{n_j}^{(j)}$ for the space $\H_j^H$ of $H$-stable vectors in $\H_j$. Then let
\[ v = \frac{1}{\sqrt{|X|}} \left\{ \sqrt{m_j} u_i^{(j)} \right\}_{j \in D,\, 1 \leq i \leq n_j} \in \bigoplus_{j\in D} \H_j^{(n_j)}. \]
It is clearly stabilized by $H$. Comparing \eqref{eq:GD} and \eqref{eq:spd}, we see that $\Phi_D := \{ \rho(g) v\}_{gH \in G/H}$ is a $(G,H)$-frame with Gram matrix $\G_D$.

When $n_j = 1$ (for instance, when $(G,H)$ is a Gelfand pair), there is an easy way to find a spanning vector for $\H_j^H$. Indeed, we can start with any $v\in \H_j$, and then the vector
\[ \tilde{v}:=\sum_{h\in H} \pi_j(h) v \]
will be stabilized by $H$. As long as $v$ is not orthogonal to $\H_j^H$, $\tilde{v}$ will be nonzero. In particular, if we apply this procedure to an entire orthonormal basis for $\H_j$, then we are guaranteed to find at least one nonzero $H$-stable vector, which we can rescale and use as $u_1^{(j)}$.
\end{example}

\begin{example} \label{ex:hrm2}
Let $G$ be a finite group acting on itself by left translation, as in Example~\ref{ex:hrm}. In this case, Theorem~\ref{thm:hgf} tells us that positive semidefinite $G$-circulant matrices are precisely the Gram matrices of $G$-frames. For \emph{projections} in the adjacency algebra and \emph{tight} $G$-frames, this is a theorem of Vale and Waldron \cite{VW05}.

If $G$ is abelian, the projections in $\A$ are determined by subsets $D\subset \hat{G}$ as in \eqref{eq:hrm}. As the reader can verify, $\G_D$ is the Gram matrix of the \emph{harmonic frame} whose synthesis matrix is made by extracting the rows indexed by $D$ from the $\hat{G} \times G$ discrete Fourier transform (DFT) matrix 
\[ \mathcal{F} = \frac{1}{\sqrt{|G|}} \begin{bmatrix} \alpha(g) \end{bmatrix}_{\alpha \in \hat{G}, g \in G}. \]

Harmonic frames were an early and abundant source of ETFs \cite{SH,XZG,DF}. Conditions for equiangularity are completely understood in terms of combinatorics. Namely, $\G_D$ represents an ETF if and only if $D$ is a \emph{difference set} in $\hat{G}$, meaning there is a constant $\lambda$ such that
\[ \left| \{ (\alpha,\beta) \in D \times D : \alpha \beta^{-1} = \gamma \} \right| = \lambda \]
for all $\gamma \neq 1_{\hat{G}}$ in $\hat{G}$.
\end{example}

\begin{example} \label{ex:cgf}
Let $K$ be any finite group, and let $\A$ be the adjacency algebra of the conjugacy class scheme described in Example~\ref{ex:gs}. Tight frames with Gram matrices in $\A$ are called \emph{central} $K$-frames in \cite{VW08}. In this case, Theorem~\ref{thm:hgf} tells us that central $K$-frames are equivalent to tight $(K\times K, \Delta(K))$-frames. 

The authors investigated conditions for equiangularity in \cite{IJM}. Briefly, $\hat{K}$ has the structure of a \emph{hypergroup}, which is a probabilistic generalization of a group, and $\G_D$ represents an ETF if and only if $D\subset \hat{K}$ is a \emph{hyperdifference set}, which is a corresponding generalization of a difference set. An infinite family of ETFs with this form, with $K$ nonabelian, appears in~\cite{IJM}.
\end{example}

Any ETF made by one of the Gelfand pairs in Examples~\ref{ex:hrm} and \ref{ex:gs} is bound by an integrality constraint: if it consists of $n$ vectors in $\C^d$ or $\R^d$, then $n-1$ must divide ${d(d-1)}$. For the harmonic frames in Example~\ref{ex:hrm}, this is an easy consequence of the difference set condition. For the central group frames in Example~\ref{ex:gs}, this was proved in \cite{IJM}. As the following example demonstrates, ETFs from more general Gelfand pairs enjoy greater latitude.

\begin{example} \label{ex:7x28}
Let $G = AGL(\F_2^3) = \F_2^3 \rtimes GL(\F_2^3)$ be the affine linear group over $\F_2^3$, acting transitively on the set of lines in $\F_2^3$ in the natural way. Using the computer program GAP \cite{gap} with package FinInG \cite{fining}, we check that this action is multiplicity free. One of its primitive idempotents is the matrix in Figure~\ref{fig:7x28}, which describes an ETF of 28 vectors in $\R^7$. The exact code used to produce this example is available online \cite{github2}.

\begin{figure}
\begin{center}
\resizebox{.9\hsize}{!}{ 
$\displaystyle
\frac{1}{12}
\left[ \begin{array}{llllllllllllllllllllllllllll}
3 & - & - & - & + & - & + & - & - & + & - & + & + & + & - & - & + & + & - & - & + & + & - & - & - & - & + & + \\
- & 3 & - & - & + & - & + & - & + & - & + & - & - & - & + & + & - & - & + & + & + & + & - & - & + & + & - & - \\
- & - & 3 & - & - & + & - & + & + & - & + & - & + & + & - & - & + & + & - & - & - & - & + & + & + & + & - & - \\
- & - & - & 3 & - & + & - & + & - & + & - & + & - & - & + & + & - & - & + & + & - & - & + & + & - & - & + & + \\
+ & + & - & - & 3 & - & - & - & + & + & - & - & + & - & + & - & - & + & - & + & + & + & - & - & + & - & - & + \\
- & - & + & + & - & 3 & - & - & + & + & - & - & + & - & + & - & + & - & + & - & - & - & + & + & - & + & + & - \\
+ & + & - & - & - & - & 3 & - & - & - & + & + & - & + & - & + & + & - & + & - & + & + & - & - & - & + & + & - \\
- & - & + & + & - & - & - & 3 & - & - & + & + & - & + & - & + & - & + & - & + & - & - & + & + & + & - & - & + \\
- & + & + & - & + & + & - & - & 3 & - & - & - & + & - & + & - & + & - & - & + & + & - & - & + & + & + & - & - \\
+ & - & - & + & + & + & - & - & - & 3 & - & - & + & - & + & - & - & + & + & - & - & + & + & - & - & - & + & + \\
- & + & + & - & - & - & + & + & - & - & 3 & - & - & + & - & + & - & + & + & - & - & + & + & - & + & + & - & - \\
+ & - & - & + & - & - & + & + & - & - & - & 3 & - & + & - & + & + & - & - & + & + & - & - & + & - & - & + & + \\
+ & - & + & - & + & + & - & - & + & + & - & - & 3 & - & - & - & + & + & - & - & - & + & - & + & - & + & - & + \\
+ & - & + & - & - & - & + & + & - & - & + & + & - & 3 & - & - & + & + & - & - & + & - & + & - & + & - & + & - \\
- & + & - & + & + & + & - & - & + & + & - & - & - & - & 3 & - & - & - & + & + & + & - & + & - & + & - & + & - \\
- & + & - & + & - & - & + & + & - & - & + & + & - & - & - & 3 & - & - & + & + & - & + & - & + & - & + & - & + \\
+ & - & + & - & - & + & + & - & + & - & - & + & + & + & - & - & 3 & - & - & - & + & - & - & + & - & + & + & - \\
+ & - & + & - & + & - & - & + & - & + & + & - & + & + & - & - & - & 3 & - & - & - & + & + & - & + & - & - & + \\
- & + & - & + & - & + & + & - & - & + & + & - & - & - & + & + & - & - & 3 & - & - & + & + & - & - & + & + & - \\
- & + & - & + & + & - & - & + & + & - & - & + & - & - & + & + & - & - & - & 3 & + & - & - & + & + & - & - & + \\
+ & + & - & - & + & - & + & - & + & - & - & + & - & + & + & - & + & - & - & + & 3 & - & - & - & + & - & + & - \\
+ & + & - & - & + & - & + & - & - & + & + & - & + & - & - & + & - & + & + & - & - & 3 & - & - & - & + & - & + \\
- & - & + & + & - & + & - & + & - & + & + & - & - & + & + & - & - & + & + & - & - & - & 3 & - & + & - & + & - \\
- & - & + & + & - & + & - & + & + & - & - & + & + & - & - & + & + & - & - & + & - & - & - & 3 & - & + & - & + \\
- & + & + & - & + & - & - & + & + & - & + & - & - & + & + & - & - & + & - & + & + & - & + & - & 3 & - & - & - \\
- & + & + & - & - & + & + & - & + & - & + & - & + & - & - & + & + & - & + & - & - & + & - & + & - & 3 & - & - \\
+ & - & - & + & - & + & + & - & - & + & - & + & - & + & + & - & + & - & + & - & + & - & + & - & - & - & 3 & - \\
+ & - & - & + & + & - & - & + & - & + & - & + & + & - & - & + & - & + & - & + & - & + & - & + & - & - & - & 3 \\
\end{array}\right] $ }
\end{center}
\caption{A primitive idempotent for the action of $AGL(\F_2^3)$ on the set of lines in $\F_2^3$. It describes a $7\times 28$ real ETF.}  \label{fig:7x28}
\end{figure}

\end{example}

\subsection{More conditions for homogeneity} \label{subsec:homCon}

Let $\Phi = \{ \phi_x \}_{x\in X}$ be any finite frame, regardless of its origin. In practice, we often find nice frames by accident and then work very hard to ``reverse engineer'' them, in the hope of finding a larger family. If we think that symmetry might play a role in our accident, then it is natural to wonder if $\Phi$ might be homogeneous, and if so, what the groups involved might be. We now address this problem. 

For each permutation $\sigma \in S(X)$, let $P_\sigma = [ \delta_{x,\sigma y}]_{x,y\in X}$ be the corresponding $X\times X$ permutation matrix. Writing $\G = \Phi^*\Phi$ for the Gram matrix of $\Phi$, we define
\[ G(\Phi) = \{ \sigma \in S(X) : P_\sigma \G = \G P_\sigma \}, \]
which is clearly a subgroup of $S(X)$. Following the proof of \cite[Lemma~3.5]{VW10}, it is easy to show that $\sigma \in G(\Phi)$ if and only if there is a unitary $U_\sigma$, necessarily unique, such that $U_\sigma \phi_x = \phi_{\sigma x}$ for all $x\in X$. When $\Phi$ is tight, it follows by \cite[Example~1]{VW10} that $G(\Phi)$ is precisely the \emph{symmetry group} of $\Phi$ introduced by Vale and Waldron \cite{VW10}. 

When restricted to \emph{tight} frames, the first part of the following theorem recasts \cite[Theorem~4.8]{VW05} in the language of Schurian association schemes.

\begin{theorem} \label{thm:sym}
Fix any point $x_0 \in X$, and let $H(\Phi) = G(\Phi)_{x_0}$ be the stabilizer of $x_0$ in $G(\Phi)$. Then the following hold:
\begin{enumerate}[(i)]
\item $\Phi$ is a homogeneous frame if and only if $G(\Phi)$ acts transitively on $X$. In that case, $\Phi$ is a $(G(\Phi),H(\Phi))$-frame.
\item If $\Phi$ is a $(G,H)$-frame for any Gelfand pair $(G,H)$, then $(G(\Phi),H(\Phi))$ is also a Gelfand pair.
\end{enumerate}
\end{theorem}

\begin{proof}
First, suppose that $G(\Phi)$ is a transitive permutation group, and let
\[ \A = \{ M \in \Mat_{X\times X}(\C) : P_\sigma M = M P_\sigma \text{ for all }\sigma \in G(\Phi) \} \]
be the corresponding adjacency algebra. We have $\G \in \A$ by definition, so $\Phi$ is a $(G(\Phi),H(\Phi))$-frame by Theorem~\ref{thm:hgf}.

Conversely, if $\Phi$ is a $(G,H)$-frame for some group $G$ and subgroup $H$, then there is a transitive group action $\alpha \colon G \to S(X)$ with point stabilizer $H$ for which $\G$ lies in the adjacency algebra
\[ \A' = \{ M \in \Mat_{X\times X}(\C) : P_\sigma M = M P_\sigma \text{ for all }\sigma \in \alpha(G) \}, \]
by Theorem~\ref{thm:hgf}. Consequently, $\alpha(G) \subset G(\Phi)$. It follows that $G(\Phi)$ is also transitive. Moreover, $\A \subset \A'$, so $\A$ is commutative whenever $\A'$ is. When $(G,H)$ is a Gelfand pair, this means that $(G(\Phi),H(\Phi))$ is, too.
\end{proof}

\begin{cor}
A finite frame $\Phi = \{ \phi_x \}_{x\in X}$ is homogeneous if and only if, for every $x,y \in X$, there is a unitary $U$ such that $U\Phi$ is a permutation of $\Phi$ with $U \phi_x = \phi_y$.
\end{cor}

\begin{rem}
We mention two more tricks that might help a researcher find a group that is ``responsible'' for a given frame $\Phi$. First, arguing just as in the proof of Theorem~\ref{thm:sym}, one can easily see that $\G$ is $G$-stable for any transitive subgroup $G\subset G(\Phi)$. Thus, if it happens that $G(\Phi)$ acts transitively and multiplicity-freely on $X$, one can search the subgroup lattice of $G(\Phi)$ to find more groups with that description. Both of those properties are preserved by inclusion into a larger permutation group, so it could be that some subgroup $G\subset G(\Phi)$ provides a starting point for a simpler description of $\Phi$, and that $G(\Phi)$ inherited those properties only by default.

Second, if one can establish that $\G$ lies in the adjacency algebra of some association scheme $\mathfrak{X}$, then one can use graph-theoretic algorithms to determine whether or not $\mathfrak{X}$ is a Schurian scheme, and if so, what the corresponding permutation group is. See \cite[Theorem~2.5]{H11} and \cite[Theorem~6.3.1]{Zie05}. We have written a GAP function that does precisely this \cite{github2}, returning the relevant permutation group $G\subset S(X)$ when it exists. When this procedure succeeds, we necessarily have $G \subset G(\Phi)$. Just as above, one can then dig deeper into the subgroup lattice of $G$ to look for alternative ``explanations'' for $\Phi$.
\end{rem}

\subsection{Regular subgroups} \label{subsec:reg}
Now consider the reverse problem: instead of starting with a frame and trying to find a group responsible for it, we start with a group $K$ (like the Heisenberg group) and try to isolate a nice set of $K$-frames. As we now explain, this is sometimes possible using an action of a \emph{larger} group $G \supset K$.

\begin{defn}
Let $G$ be a finite group acting transitively on a set $X$. A subgroup $K\subset G$ is called \emph{regular} for this action if:
\begin{enumerate}[(1)]
\item The action of $K$ on $X$ is transitive; and
\item The stabilizer $K_x$ of each $x\in X$ in $K$ is trivial.
\end{enumerate}
\end{defn}
For the remainder of this subsection, we fix a finite group $G$ acting transitively on a set $X$, and we let $H$ denote the stabilizer of a fixed base point $x_0\in X$. It is easy to prove that a subgroup $K\subset G$ is regular if and only if $G= KH$ and $K \cap H = \{1_G\}$, if and only if $K$ is a complete and irredundant set of left coset representatives for $H$ in $G$. This leads us to the following simple, but crucial, observation.

\begin{prop} \label{prop:milk}
If $K\subset G$ is a regular subgroup for the action of $G$ on $X$, then every $(G,H)$-frame is a $K$-frame.
\end{prop}

\begin{proof}
The expression $\Phi = \{ \rho(g) v \}_{gH \in G/H}$ for a $(G,H)$-frame does not depend on the choice of coset representatives $g\in gH$, so $\Phi = \{ \rho(k) v \}_{k \in K}$.
\end{proof}

\begin{rem}
It may very well happen that $(G,H)$ is a Gelfand pair, while a regular subgroup $K$ is nonabelian. In that case, the spherical functions for the Gelfand pair distinguish a finite subset of the uncountably infinite collection of tight $K$-frames. (See \cite{VW08} for the cardinality of $K$-frames.)

An important class of examples occurs when $G = K \rtimes H$ and $(G,H)$ forms a Gelfand pair. For the researcher exploring the class of $K$-frames, it may be worthwhile to traverse the lattice of subgroups $H \subset \Aut(K)$ and examine the spherical functions formed by any Gelfand pairs $({K\rtimes H},H)$. In fact, that is exactly how the authors discovered the class of Heisenberg ETFs described in Section~\ref{sec:Heis}.
\end{rem}

We can say more by involving the adjacency algebra. When $K\subset G$ is regular for the action of $G$ on $X$, the mapping $k\mapsto k\cdot x_0$ determines a bijection $K \cong X$, which turns $K$ into a $G$-set with the action
\[ g\cdot k = k' \iff gk\cdot x_0 = k'\cdot x_0 \qquad (g\in G;\ k,k' \in K), \]
or equivalently,
\begin{equation} \label{eq:gk}
g\cdot k = k' \iff gkh = k' \text{ for some $h \in H$} \qquad (g\in G;\ k,k' \in K).
\end{equation}
If it happens that $H$ normalizes $K$ (for instance, if $G = K \rtimes H$), then the action of $H$ on $K$ is simply conjugation; however, we do not assume this in general. 
The theorem below says that we can identify $\A$ with the space
\[ L^2(K)^H := \{ \varphi \in L^2(K) : \varphi(h \cdot k) = \varphi(k) \text{ for all $k \in K$ and $h\in H$}\}. \]
This appears to be a folk theorem. The closest reference we could locate is \cite[Theorem~6.1]{BI}. We include a proof here for completeness.

\begin{theorem} \label{thm:milk}
When $K\subset G$ is a regular subgroup, $L^2(K)^H$ is a $*$-subalgebra of $L^2(K)$. It is isomorphic to the adjacency algebra of $G$-stable matrices by mapping $A\in \A$ to the function 
\[ \tilde{\varphi}_A(k) = A_{x_0,k\cdot x_0} \qquad (k\in K). \]
\end{theorem}

\begin{proof}
The main idea is to identify $L^2(K)^H$ with $L^2(H\backslash G /H)$. We will write $[k]$ for the $H$-orbit of $k\in K$, and $K^H$ for the set of $H$-orbits in $K$, with a similar notation in $G/H$. From \eqref{eq:gk}, we see that
\[ [k] = HkH \cap K \qquad (k \in K). \]
The mapping $K \to G/H$, $k \mapsto kH$, is an isomorphism of $H$-sets, so $K^H \cong (G/H)^H$ by the bijection $[k] \mapsto [kH]$ . On the other hand, $H$-orbits in $G/H$ are precisely described by double cosets through the correspondence $[gH] \mapsto HgH$. Consequently, $H\backslash G / H \cong K^H$ through the mapping $HgH \mapsto HgH \cap K$. It follows that there is a linear isomorphism $T\colon L^2(H \backslash G / H) \to L^2(K)^H$ with $T \mathbf{1}_{HgH} = |H| \cdot \mathbf{1}_{HgH \cap K}$ for all $g\in G$; or equivalently,
\[ (T \mathbf{1}_{HgH})(k) = |H|\cdot \mathbf{1}_{HgH \cap K}(k) = |H| \cdot \mathbf{1}_{HgH}(k) \qquad (k \in K). \]
Extending linearly, we see that $(T \psi)(k) = |H|\cdot \psi(k)$ for all $\psi \in L^2(H\backslash G / H)$ and $k\in K$.

We claim that $T$ preserves the $*$-algebraic structure of $L^2(H\backslash G / H)$. Let us write $\psi_1 *_G \psi_2$ for convolution in $L^2(G)$ and $\varphi_1 *_K \varphi_2$ for convolution in $L^2(K)$. For any $\psi_1,\psi_2 \in L^2(H\backslash G / H)$ and any $k' \in K$, we have
\begin{align*}
[T(\psi_1 *_G \psi_2)](k') &= |H|\cdot \sum_{g\in G} \psi_1(g) \psi_2(g^{-1} k').
\end{align*}
Any $g\in G$ can be written uniquely in the form $g= k_g h_g$ with $k_g \in K$ and $h_g \in H$, so
\begin{align*}
[T(\psi_1 *_G \psi_2)](k') &= |H| \cdot \sum_{g\in G} \psi_1(k_g h_g) \psi_2(h_g^{-1} k_g^{-1} k') \\[5 pt]
&= |H| \cdot \sum_{g\in G} \psi_1(k_g) \psi_2(k_g^{-1} k').
\end{align*}
Reindexing with $g = kh$ for $k \in K$ and $h \in H$, we obtain
\begin{align*}
[T(\psi_1 *_G \psi_2)](k') &= |H|\cdot \sum_{k \in K} \sum_{h\in H}  \psi_1(k) \psi_2(k^{-1} k') \\[5 pt]
&= |H|^2 \cdot \sum_{k \in K} \psi_1(k) \psi_2(k^{-1} k') \\[5 pt]
&= \sum_{k \in K} (T \psi_1)(k) \cdot (T \psi_2)(k^{-1} k') \\[5 pt]
&= [(T \psi_1) *_K (T \psi_2)](k').
\end{align*}
In other words,
\[ T(\psi_1 *_G \psi_2) = (T \psi_1) *_K (T \psi_2). \]
It is easy to see that $T$ preserves the involution. Therefore, $L^2(K)^H$ is a $*$-subalgebra of $L^2(K)$, and $T\colon L^2(H\backslash G / H) \to L^2(K)^H$ is a $*$-algebra isomorphism.

To complete the proof, we recall that $\A$ is isomorphic to $L^2(H\backslash G/H)$ by mapping $A\in \A$ to the function
\[ \tilde{\psi}_A(g) = \psi_A(g^{-1}) = \frac{1}{|H|} A_{g^{-1}\cdot x_0, x_0} = \frac{1}{|H|} A_{x_0,g\cdot x_0} \qquad (g \in G). \]
Now, we simply observe that $T \tilde{\psi}_A = \tilde{\varphi}_A$.
\end{proof}

With the aid of Theorem~\ref{thm:milk}, we can now say precisely which $K$-frames arise as $(G,H)$-frames.

\begin{defn}
Let $\rho \colon K \to U(\H)$ be a unitary representation of $K$, and suppose there is a vector $v\in \H$ for which $\Phi:=\{\rho(k)v\}_{k\in K}$ is a frame. The \emph{function of positive type} associated with $\Phi$ is $\varphi \colon K \to \C$, given by
\[ \varphi(k) = \langle v, \rho(k) v \rangle \qquad (k\in K). \]

When $K\subset G$ is a regular subgroup, we say that $\Phi$ \emph{extends to a $(G,H)$-frame} if there is an extension $\tilde{\rho} \colon G \to U(\H)$ of $\rho$ for which $v$ is an $H$-stable vector; in that case, $\Phi = \{ \tilde{\rho}(g) v\}_{gH \in G/H}$ is literally a $(G,H)$-frame .
\end{defn}

Functions of positive type classify $K$-frames up to unitary equivalence; see \cite{I2} for details.

\begin{cor} \label{cor:milk}
When $K\subset G$ is a regular subgroup, a $K$-frame $\Phi = \{ \rho(k) v \}_{k\in K}$ extends to a $(G,H)$-frame if and only if its function of positive type belongs to $L^2(K)^H$.
\end{cor}

\begin{proof}
Let $\varphi$ be the function of positive type associated with $\Phi$, and let $\tilde{\varphi} = \overline{ \varphi^* }$ be the function
\[ \tilde{\varphi}(k) = \langle \rho(k) v, v \rangle \qquad (k \in K). \]
We will prove the statement with $\tilde{\varphi}$ in place of $\varphi$; this is equivalent since $\varphi$ and $\overline{\varphi^*}$ lie in $L^2(K)^H$ at precisely the same time.

The Gram matrix $\G$ of $\Phi$ has entries
\[ \G_{k_1,k_2} = \langle \rho(k_2) v, \rho(k_1) v \rangle = \langle \rho(k_1^{-1} k_2) v, v \rangle = \tilde{\varphi}(k_1^{-1} k_2) \qquad (k_1,k_2 \in K). \]
If there is a representation $\tilde{\rho} \colon G \to U(\H)$ extending $\rho$ for which $v$ is $H$-stable, then we can view $\Phi$ as a $(G,H)$-frame. After reindexing with the identification $K\cong X$, $\G$ becomes a matrix $\G' \in \A$, by Theorem~\ref{thm:hgf}. For any $k\in K$, we have $\tilde{\varphi}(k) = \G_{1_K,k} = \tilde{\varphi}_{\G'}(k)$, so $\tilde{\varphi} = \tilde{\varphi}_{\G'} \in L^2(K)^H$.

Conversely, when $\tilde{\varphi} \in L^2(K)^H$, we can find a matrix $\G' \in \A$ with $\tilde{\varphi}_{\G'} = \tilde{\varphi}$, so that
\[ \G'_{k_1\cdot x_0, k_2 \cdot x_0} = \G'_{x_0, k_1^{-1} k_2\cdot x_0} = \tilde{\varphi}(k_1^{-1} k_2) = \G_{k_1,k_2} \qquad (k_1,k_2 \in K). \]
In particular, $\G'$ is positive semidefinite. By Theorem~\ref{thm:hgf} again, there is a unitary representation $\rho' \colon G \to U(\H')$ and an $H$-stable vector $v' \in \H'$ for which $\Phi':= \{ \rho'(k) v' \}_{k\in K}$ is a frame with Gram matrix $\G$. Since $\Phi$ and $\Phi'$ have the same Gram matrix, there is a unitary $U\colon \H' \to \H$ with
\[ U \rho'(k)v' = \rho(k) v \qquad (k \in K). \]
Define $\tilde{\rho} \colon G \to U(\H)$ by setting
\[ \tilde{\rho}(g) = U \rho'(g) U^{-1} \qquad (g\in G). \]
For any $h\in H$, we have
\[ \tilde{\rho}(h) v = U \rho'(h) v' = U v' = v, \]
so $v$ is $H$-stable under $\tilde{\rho}$. 

It only remains to show that $\tilde{\rho}$ extends $\rho$. Since $\Phi$ is a frame, any vector $w\in \H$ can be written in the form
\[ w = \sum_{k\in K} c_k \rho(k) v \]
for some constants $c_k \in \C$.  For any $k,k' \in K$, we have
\[ \tilde{\rho}(k') \rho(k) v = U \rho'(k') \rho'(k) v' = U \rho'(k' k) v = \rho(k' k) v, \]
and therefore,
\[ \tilde{\rho}(k') w = \sum_{k\in K} c_k \tilde{\rho}(k') \rho(k) v = \sum_{k\in K} c_k \rho(k' k) v = \rho(k') w. \]
In other words, $\tilde{\rho}(k') = \rho(k')$, as desired.
\end{proof}

\begin{example}
Let $K$ be any finite group, and let $H = \Inn(K)$ be the group of inner automorphisms, given by conjugation in $K$. Then $K$ is a regular subgroup for the action of $G:=K\rtimes H$ on $G/H$, and $L^2(K)^H$ consists of functions constant on conjugacy classes in $K$. This is well known to be the center of the $*$-algebra $L^2(K)$. By Theorem~\ref{thm:milk}, the adjacency algebra of this action is commutative, and $(K \rtimes \Inn(K),\Inn(K))$ is a Gelfand pair.

If $\Phi = \{ \rho(k) v \}_{k\in K}$ is any $K$-frame, with Gram matrix $\G$ and function of positive type $\varphi$, then the reader can easily verify that $\varphi \in L^2(K)^{\Inn(K)}$ if and only if $\G$ is stable under the action of $K\times K$ described in Example~\ref{ex:cgf}. As explained in that example, this means precisely that $\G$ is the Gram matrix of a central $K$-frame. Applying Corollary~\ref{cor:milk}, we see that central $K$-frames are the same as tight $(K \rtimes \Inn(K),\Inn(K))$-frames.
\end{example}

\subsection{Projective reduction} \label{subsec:proj}

We conclude this section with a brief investigation of the following technique, applied in the setting of homogeneous frames.

\begin{defn}\label{def:PR}
Given a frame $\Phi=\{\phi_x\}_{x\in X}$, we introduce an equivalence relation on $X$ by saying that $x \sim y$ if there is a unimodular constant $\alpha_{x,y}$ such that $\phi_y = \alpha_{x,y} \phi_x$. A \emph{projective reduction} of $\Phi$ consists of one vector $\phi_x$ for each equivalence class $[x] \in X/\sim$.
\end{defn}

This concept is undoubtedly familiar to anyone who has studied SIC-POVMs \cite{Z}, since the full orbit of a vector under the Schr\"{o}dinger representation of the Heisenberg group contains several ``projective copies'' of the same frame, and has to be projectively reduced before it can possibly make a SIC-POVM.

Even if we only know the Gram matrix of $\Phi$, we can still find the Gram matrix of a projective reduction, since $\phi_x \sim \phi_y$ if and only if the $y$-th column of $\Phi^* \Phi$ is a unimodular multiple of the $x$-th column.

\begin{defn} \label{def:PS}
Let $G$ be a finite group, and let $\rho\colon G \to U(\H)$ be a unitary representation. The \emph{projective stabilizer} of a vector $v\in \H$ is the group
\[ K = \{ g \in G : \rho(g) v = \alpha(g) v \text{ for some constant }\alpha(g) \in \T\}. \]
\end{defn}

The name derives from the action of $G$ on one-dimensional subspaces of $\H$; when $v\neq 0$, $K$ is precisely the stabilizer of $\spn\{v\}$. Equivalently, $K$ is the subgroup of all $g\in G$ for which $v$ is an eigenvector of $\rho(g)$. If $\Phi = \{ \rho(g) v \}_{g\in G}$ is a $G$-frame and $\varphi$ is its function of positive type, then the conditions for equality in Cauchy-Schwarz easily imply that the projective stabilizer of $v$ is
\begin{equation} \label{eq:PS}
K = \{ g\in G : | \varphi(g) | = \varphi(1_G) \}.
\end{equation}

\begin{prop} \label{prop:PS}
Let $\Phi = \{ \rho(g) v\}_{g\in G}$ be a $G$-frame whose generator $v$ has projective stabilizer $K$. Choose representatives $g_1,\dotsc,g_n \in G$ for the left cosets of $K$ in $G$, and let $\Phi_{\text{red}} = \{ \rho(g_j) v\}_{j=1}^n$. Then $\Phi_{\text{red}}$ is a projective reduction of $\Phi$. Moreover, every projective reduction of $\Phi$ takes this form.
\end{prop}

\begin{proof}
Let $\sim$ be the equivalence relation from Definition~\ref{def:PR}. For $g,h\in G$, we have $g\sim h$ if and only if there is a unimodular constant $\alpha_{g,h} \in \T$ such that $\rho(g) v = \alpha_{g,h} \rho(h) v$, or equivalently, $\rho(h^{-1} g) v = \alpha_{g,h} v$. That happens if and only if $h^{-1}g \in K$. Thus, the equivalence classes for $\sim$ are precisely the left cosets of $K$ in $G$.
\end{proof}

For general frames, it is easy to construct examples where projective reduction fails to preserve tightness. For group frames, however, the situation is much nicer.

\begin{prop} \label{prop:PRG}
Let $\Phi = \{ \rho(g) v \}_{g\in G}$ be a $G$-frame whose generator $v$ has projective stabilizer $K$. Let $H$ be a subgroup of $K$, and let $g_1,\dotsc,g_n \in G$ be left coset representatives for $H$ in $G$. Then $\Phi$ has frame bounds $A,B$ if and only if $\Phi' := \{ \rho(g_j) v\}_{j=1}^n$ has frame bounds $A/|H|,B/|H|$. In particular, $\Phi$ is tight if and only if $\Phi'$ is, too.
\end{prop}

Consequently, the projective reduction of a tight group frame is tight, and vice versa.

\begin{proof}
Let $\alpha \colon K \to \T$ be as in Definition~\ref{def:PS}. For any $w\in \H$, we have
\begin{align*}
\sum_{g\in G} | \langle w, \rho(g) v \rangle |^2 &= \sum_{j=1}^n \sum_{h \in H} | \langle w, \rho(g_j h) v \rangle |^2 = \sum_{j=1}^n \sum_{h \in H} | \langle w, \rho(g_j) \alpha(h) v \rangle |^2 \\
&= |H| \sum_{j=1}^n |\langle w, \rho(g_j) v \rangle |^2,
\end{align*}
and the proposition follows immediately.
\end{proof}

\begin{cor}
Let $\A$ be the adjacency algebra of a Schurian association scheme. Let $\mathcal{G} \in \A$, and let $\Phi$ be a frame with $\Phi^* \Phi = \mathcal{G}$. If $\Phi$ is tight, then so is its projective reduction, and vice versa.
\end{cor}

\begin{proof}
Let $G$ be the permutation group that produced the association scheme, and let $H \subset G$ be the stabilizer of a point. By Theorem~\ref{thm:hgf}, we may assume without loss of generality that $\Phi = \{ \rho(g) v \}_{gH \in G/H}$ for some unitary representation $\rho$ and some $H$-stable vector $v$, whose projective stabilizer necessarily contains $H$ as a subgroup. Let $\tilde{\Phi} = \{ \rho(g) v \}_{g\in G}$ be the full orbit of $v$. By Proposition~\ref{prop:PRG}, $\Phi$ is tight if and only if $\tilde{\Phi}$ is, too. Moreover, any projective reduction $\Phi_{\text{red}}$ of $\Phi$ is also a projective reduction of $\tilde{\Phi}$. Using Proposition~\ref{prop:PS}, it follows that $\Phi$, $\tilde{\Phi}$, and $\Phi_{\text{red}}$ are tight at the same time.
\end{proof}

\begin{cor}
Let $G$ be a finite group acting transitively on a set $X$, with notation as in Section~\ref{sec:pre}. Given any subset $D \subset \{0,\dotsc,r\}$, let $\Phi_D$ be the frame constructed in Example~\ref{ex:gphg}, and let
\[ \varphi_D := \frac{1}{|X|} \sum_{j\in D} m_j \omega_j. \]
Then the following are equivalent:
\begin{enumerate}[(i)]
\item Any projective reduction of $\Phi_D$ is an equiangular tight frame.
\item $\varphi_D$ takes exactly two absolute values.
\end{enumerate}
\end{cor}

\begin{proof}
As explained in Example~\ref{ex:gphg}, the Gram matrix of $\Phi_D$ has entries 
\[ (\G_D)_{gH,hH} = \varphi_D(h^{-1} g) \qquad (g,h \in G). \]
Thus, $\varphi_D$ is the function of positive type associated with the $G$-frame $\tilde{\Phi}_D := \{ \rho(g) v \}_{g\in G}$.

Following Proposition~\ref{prop:PS}, we let
\[ K = \left\{ g \in G : |\varphi_D(g)| = \varphi_D(1_G) \right\} \]
be the projective stabilizer of $v$, and choose representatives $1_G = g_1,\dotsc,g_m\in G$ for the left cosets of $K$ in $G$. Now $\Phi_D^{\text{red}}:=\{\rho(g_i) v\}_{i=1}^m$ is a projective reduction of $\tilde{\Phi}_D$, hence also of $\Phi_D$. It is tight by Proposition~\ref{prop:PRG}, and its Gram matrix $\G_D^{\text{red}}$ has entries
\[ (\G_D^{\text{red}})_{i,j} = \langle \rho(g_j) v, \rho(g_i) v \rangle = \langle v, \rho(g_j^{-1} g_i) v \rangle = \varphi_D(g_j^{-1} g_i) \qquad (1 \leq i,j \leq m). \]

If $\varphi_D$ takes exactly two absolute values, then $\left|(\G_D^{\text{red}})_{i,j}\right|$ is constant across all pairs $i \neq j$, since $g_j^{-1} g_i \notin K$. Thus, $\Phi_D^{\text{red}}$ is an ETF. On the other hand if $|\varphi_D(g)| \neq | \varphi_D(g')|$ for some $g,g' \in G \setminus K$, then we can find $i,j \neq 1$ and $k,k' \in K$ for which $g_i k$ and $g' = g_j k'$, so that
\[ |\langle \rho(g_1) v, \rho(g_i) v \rangle | = | \langle v, \rho(g_i k) v \rangle | = | \varphi_D(g)| \neq | \varphi_D(g')| = | \langle \rho(g_1) v, \rho(g_j) v \rangle |, \]
and $\Phi_D^{\text{red}}$ is not equiangular.
\end{proof}

\section{Examples}
Having established the basic theory of homogeneous frames, we now give several examples to demonstrate the fruitfulness of this approach for constructing optimal line packings. In every example below, the bulk of the work lies in identifying a transitive action of a finite group (or equivalently, a subgroup corresponding to its stabilizer). Once this is done, the spherical functions identify a handful of line packings for us, and we just have to pick out the ones with low coherence. The code that produced these examples is available online \cite{github2}.

\begin{example}
Let
\[ T = \begin{bmatrix}
0 & 1 \\
1 & 0
\end{bmatrix} \qquad \text{and} \qquad  M = \begin{bmatrix}
1 & 0 \\
0 & -1
\end{bmatrix}, \]
and let
\[ K = \{ i^k T^{t_1} M^{m_1} \otimes T^{t_2} M^{m_2} \otimes T^{t_3} M^{m_3} : 0 \leq k \leq 3,\ 0 \leq t_i, m_i \leq 1 \} \]
be the group generated by $iI_8$ and all three-fold tensor products of powers of $T$ and $M$. This is a slightly enlarged version of the three-qubit Weyl-Heisenberg group. Zauner \cite{Z} observed that the vector
\[ v = \frac{1}{\sqrt{6}} \begin{bmatrix} 1+i, & 0, & -1, & 1, & -i, & -1, & 0, & 0 \end{bmatrix}^T \]
generates a $K$-frame $\Phi:= \{ g v \}_{g\in K}$ whose projective reduction is an $8\times 64$ ETF, equivalent to Hoggar's 64 lines \cite{Hog78,Hog81}. Zhu \cite{Zhu12} found that when $\omega = \exp(2\pi i /8)$, both of the unitary matrices
\[ U = \frac{\omega}{\sqrt{2}} \left[ \begin{array}{rrrrrrrr}
0 & 0 & -1 & 0 & i & 0 & 0 & 0 \\
0 & 0 & -i & 0 & 1 & 0 & 0 & 0 \\
0 & 0 & 0 & i & 0 & 1 & 0 & 0 \\
0 & 0 & 0 & 1 & 0 & i & 0 & 0 \\
-1 & 0 & 0 & 0 & 0 & 0 & i & 0 \\
i & 0 & 0 & 0 & 0 & 0 & -1 & 0 \\
0 & i & 0 & 0 & 0 & 0 & 0 & 1 \\
0 & -1 & 0 & 0 & 0 & 0 & 0 & -i
\end{array} \right] \]
and
\[ V = \frac{\omega}{\sqrt{2}} \left[ \begin{array}{rrrrrrrr}
0 & 0 & 0 & 0 & i & -1 & 0 & 0 \\
0 & 0 & 0 & 0 & -i & -1 & 0 & 0 \\
i & 1 & 0 & 0 & 0 & 0 & 0 & 0 \\
-i & 1 & 0 & 0 & 0 & 0 & 0 & 0 \\
0 & 0 & i & -1 & 0 & 0 & 0 & 0 \\
0 & 0 & -i & -1 & 0 & 0 & 0 & 0 \\
0 & 0 & 0 & 0 & 0 & 0 & i & 1 \\
0 & 0 & 0 & 0 & 0 & 0 & -i & 1
\end{array} \right] \]
fix $v$. The same is therefore true of every matrix in $H := \langle U, V \rangle$, which turns out to be isomorphic to the simple group $PSU_3(\F_3)$.

Using GAP \cite{gap}, we verify that $H$ normalizes $K$ while intersecting it trivially, and that $(K\rtimes H, H)$ is a Gelfand pair. We can view $\Phi$ as a $(K\rtimes H,H)$-frame, one of the finitely many distinguished by our pair. To find the others, we compute the spherical functions:

\begin{center}
\resizebox{\hsize}{!}{
\begin{tabular}{r|rrrrrrrr|r}
$g$ & $I_8$ & $M\otimes I_2 \otimes I_2$ & $T\otimes I_2 \otimes I_2$ & $TM \otimes I_2 \otimes I_2$ & $I_2 \otimes TM \otimes I_2$ & $-I_8$ & $iI_8$ & $-iI_8$ & $m_j$ \\ \hline
$|HgH|$ & $6048$ & $381024$ & $381024$ & $381024$ & $381024$ & $6048$ & $6048$ & $6048$ & \\ \hline 
$\omega_1(g)$ & 1 & 1 & 1 & 1 & 1 & 1 & 1 & 1 & 1 \\
$\omega_2(g)$ & $1$ & $1/3$ & $-1/3$ & $i/3$ & $-i/3$ & $-1$ & $i$ & $-i$ & $8$ \\
$\omega_3(g)$ & $1$ & $1/3$ & $-1/3$ & $-i/3$ & $i/3$ & $-1$ & $-i$ & $i$ & $8$ \\
$\omega_4(g)$ & $1$ & $-1/7$ & $-1/7$ & $1/7$ & $1/7$ & $1$ & $-1$ & $-1$ & $28$ \\
$\omega_5(g)$ & $1$ & $1/9$ & $1/9$ & $-1/9$ & $-1/9$ & $1$ & $-1$ & $-1$ & $36$ \\
$\omega_6(g)$ & $1$ & $-1/21$ & $1/21$ & $-i/21$ & $i/21$ & $-1$ & $i$ & $-i$ & $56$ \\
$\omega_7(g)$ & $1$ & $-1/21$ & $1/21$ & $i/21$ & $-i/21$ & $-1$ & $-i$ & $i$ & $56$\\
$\omega_8(g)$ & $1$ & $-1/63$ & $-1/63$ & $-1/63$ & $-1/63$ & $1$ & $1$ & $1$ & $63$\\
\end{tabular}}
\end{center}

Every nontrivial spherical function takes exactly two absolute values, so it produces an ETF after projective reduction. In each case, the corresponding projective stabilizer has order $4\times 6{,}048 = 24{,}192$, while the full group has order $1{,}548{,}288$. Thus, the projective reductions have $1{,}548{,}288/24{,}192=64$ vectors, as expected. In this way, we obtain Hoggar's complex $8\times 64$ ETF, a real $28 \times 64$ ETF, and Naimark complements of both.
\end{example}

\begin{example}
The following was privately suggested to the authors by Prof.\ Henry Cohn. Let $G= SL_2(\F_8)$, acting doubly transitively on the set $S$ of one-dimensional subspaces of $\F_8^2$. Then $G$ also acts transitively on 
\[ X = \{ (\ell_1,\ell_2) \in S^2 : \ell_1 \neq \ell_2 \} \]
through the coordinate-wise action. Using GAP, we observe that exactly one constituent of the permutation character is real-valued with degree seven, and that it occurs with multiplicity one. After projective reduction, its spherical function describes $36$ unit norm vectors in $\R^7$ with coherence $1/3$, the minimum currently known for such a packing. Lines achieving this value appear in Sloane's database of putatively optimal packings \cite{Sloane}, and in \cite{FJM17}, where they are described as ``elusive''. We do not know if this arrangement is optimal.
\end{example}

\begin{example}
The Mathieu group $M_{11}$ acts triply transitively on a twelve-point set $S=\{p_1,\dotsc,p_{12}\}$, as implemented in GAP by \texttt{TransitiveGroup(12,272)}. Following the previous example, we consider the action of $M_{11}$ on
\[ X = \{ (p_i,p_j) \in S^2 : i \neq j\}. \]
Exactly one constituent of the permutation character has degree ten, and it occurs with multiplicity one. Let $\omega$ be its spherical function, and let $\omega_0$ be the trivial spherical function. Taking $D=\{\omega_0,\omega\}$ in Proposition~\ref{prop:GD}, and projectively reducing the result, we obtain a tight frame of $66$ vectors in $\R^{11}$ with coherence $1/3$. This beats the previous record for 66 lines in $\R^{11}$, as recorded in Sloane's database \cite{Sloane}. Unlike the previous record holder, this packing is a tight frame. We do not know if it is optimal. The reader can find its Gram matrix online \cite{github2}.
\end{example}

\begin{example}
Let $L$ be a Sylow 2-subgroup of $SL_2(\F_5)$, and let $G_0 = N_{GL_2(\F_5)}(L)$ be its normalizer in $GL_2(\F_5)$. Computing in GAP, we find that the affine action of $G:= \F_5^2 \rtimes G_0$ on $\F_5^2$ is doubly transitive. Once again, we let $G$ act on
\[ X = \{ (x,y) \in \F_5^2 \times \F_5^2 : x \neq y \}. \]
It turns out that if we let $K_0 = N_{SL_2(\F_5)}(L) \cong SL_2(\F_3)$, then 
\[ K := \F_5^2 \rtimes K_0 \cong \F_5^2 \rtimes SL_2(\F_3) \] 
is a regular subgroup for this action.

Among the spherical functions, there is exactly one pair of complex conjugates  $\omega$ and $\overline{\omega}$ corresponding to characters of degree two, and each occurs with multiplicity one. Taking either of $D = \{\omega\}$ or $D = \{\overline{\omega}\}$ in Proposition~\ref{prop:GD} produces, after projective reduction, six lines in $\C^2$ with coherence $1/\sqrt{2}$. On the other hand, taking $D = \{ \omega, \overline{\omega} \}$ produces 12 lines in $\R^4$ with coherence $1/2$. Both of these packings achieve equality in the orthoplex and Levenstein bounds \cite{ConwayHS:96,levenshtein1982bounds}, and are therefore optimal. In fact, both are mutually unbiased bases that occur as projective reductions of $K$-frames. Their Gram matrices are shown in Figures~\ref{fig:4x12} and~\ref{fig:2x6}.
\end{example}

\begin{figure}
\[ \frac{1}{2} \left[ \begin{array}{rrrrrrrrrrrr}
2 & 0 & 0 & 0 & - & + & + & + & - & - & - & - \\
0 & 2 & 0 & 0 & - & - & + & - & + & - & - & + \\
0 & 0 & 2 & 0 & - & - & - & + & + & + & - & - \\
0 & 0 & 0 & 2 & - & + & - & - & + & - & + & - \\
- & - & - & - & 2 & 0 & 0 & 0 & - & + & + & + \\
+ & - & - & + & 0 & 2 & 0 & 0 & - & - & + & - \\
+ & + & - & - & 0 & 0 & 2 & 0 & - & - & - & + \\
+ & - & + & - & 0 & 0 & 0 & 2 & - & + & - & - \\
- & + & + & + & - & - & - & - & 2 & 0 & 0 & 0 \\
- & - & + & - & + & - & - & + & 0 & 2 & 0 & 0 \\
- & - & - & + & + & + & - & - & 0 & 0 & 2 & 0 \\
- & + & - & - & + & - & + & - & 0 & 0 & 0 & 2 \\
\end{array} \right] \]
\caption{Three mutually unbiased bases in $\R^4$ that occur as a projectively reduced $ \F_5^2 \rtimes SL_2(\F_3)$-frame.} \label{fig:4x12}
\end{figure}

\begin{figure}
\[ \frac{1}{2} \left[ \begin{array}{rrrrrr}
2 & 0 & -1+i & 1+i & -1-i & -1-i \\
0 & 2 & -1-i & 1-i & 1+i & -1-i \\
-1-i & -1+i & 2 & 0 & -1+i & 1+i \\
1-i & 1+i & 0 & 2 & -1+i & -1-i \\
-1+i & 1-i & -1-i & -1-i & 2 & 0 \\
-1+i & -1+i & 1-i & -1+i & 0 & 2 \\
\end{array} \right] \]
\caption{Three mutually unbiased bases in $\C^2$ that occur as a projectively reduced $\F_5^2 \rtimes SL_2(\F_3)$-frame.} \label{fig:2x6}
\end{figure}

\begin{example}
GAP contains a library of all transitive permutation groups of degree $m \leq 30$, up to conjugacy within $S_m$. We made a brute force search of this library for actions that produce homogeneous ETFs. The prodigious results for real and complex ETFs appear in Tables~\ref{tbl:rETFs} and~\ref{tbl:cETFs}, respectively. In both tables, the parameters $d$ and $n$ indicate an ETF of $n$ vectors in either $\R^d$ or $\C^d$, and the remaining parameters ($m$, $t$ and $k$) describe a transitive group action that can be used to create the Gram matrix of such an ETF. Namely, let $G$ be the group created by the GAP command \texttt{G:=TransitiveGroup(m,k)}, which acts $t$-transitively on the set $S=\{1,\dotsc,m\}$. When $t=1$ we take $X=S$, and when $t=2$ we take
\[ X = \{ (i,j) \in S^2 : i \neq j \}. \]
The action of $G$ on $X$ is transitive, and some subset $D$ of the corresponding spherical functions creates a projection $\G_D$ whose projective reduction is the Gram matrix of a $d\times n$ ETF. We implemented  a GAP function to produce that Gram matrix from the parameters $(d,n,m,t,k)$; the reader can find it online \cite{github2}.

For the sake of computational efficiency, we restricted our attention to groups of order $|G| \leq 100,000$ with degree $m \neq 24$, and considered only spherical functions corresponding to constituents of the permutation character having multiplicity one. When there were more than 20 such spherical functions, we ignored the action completely.

Our results are significant for demonstrating that the spherical function construction of Proposition~\ref{prop:GD}, combined with projective reduction, accounts for all known sizes $d \times n$ of real and complex ETFs with $n \leq 30$, with four possible exceptions \cite{FMTable}.
In particular, Tables~\ref{tbl:rETFs} and~\ref{tbl:cETFs} capture all known sizes in this range except $4 \times 16$, $12\times 16$, $5 \times 25$, $20 \times 25$, $10\times25$ and $15\times 25$.
These last two are part of an infinite family that we construct later in Theorem~\ref{thm:HETF}.
While these also arise from a transitive group action, the size of the underlying set, namely $m=125>30$, was beyond the scope of our GAP search.
The remaining sizes are SIC-POVMs and their Naimark complements, themselves constructed as projective reductions of Heisenberg group frames; therefore, \emph{every} known ETF size $d \times n$ with $n \leq 30$ can be realized as a projective reduction of a homogeneous frame.
\end{example}

\setlength{\columnseprule}{0.4pt}
\begin{table}[t]
\caption{Real Homogeneous ETFs} \label{tbl:rETFs}
\tiny
\begin{center}
\begin{tabularx}{\textwidth}{rrrrX}
\footnotesize{$d$} & \footnotesize{$n$} & \footnotesize{$m$} & \footnotesize{$t$} & \footnotesize{$k$} \\ \hline
3&6&12&1&33, 76\\ 3&6&25&2&74, 81\\ 5&10&10&1&7, 13\\ 5&10&20&1&15, 30, 31, 32, 35, 36, 62, 65, 89, 149, 152, 172, 177, 198, 218, 219, 222, 223, 224, 225, 228, 280, 281, 285, 289, 291, 365, 570, 674, 676, 693\\ 5&10&30&1&45, 94, 98, 101, 174, 549, 558, 562, 717, 719, 727, 728, 731, 732, 733, 828, 829, 832, 833, 837, 844, 910, 919, 1015, 1021, 1023, 1024, 1027, 1031, 1032, 1033, 1036, 1165, 1166, 1167, 1168, 1246, 1255, 1378, 1381, 1382, 1385, 1387, 1390, 1392, 1395\\ 6&16&16&1&2, 3, 4, 19, 25, 28, 34, 46, 51, 57, 58, 61, 62, 63, 64, 109, 135, 143, 147, 166, 178, 179, 181, 182, 183, 185, 186, 191, 193, 194, 195, 395, 414, 415, 421, 425, 431, 436, 444, 708, 709, 710, 711, 748, 1030, 1033, 1034, 1081, 1294, 1328\\ 7&14&28&1&120, 199\\ 7&28&28&1&27, 32, 46, 159, 433, 502\\ 10&16&16&1&2, 3, 4, 19, 25, 28, 34, 46, 51, 57, 58, 61, 62, 63, 64, 109, 135, 143, 147, 166, 178, 179, 181, 182, 183, 185, 186, 191, 193, 194, 195, 395, 414, 415, 421, 425, 431, 436, 444, 708, 709, 710, 711, 748, 1030, 1033, 1034, 1081, 1294, 1328\\ 21&28&8&2&43, 48, 49, 50\\ 21&28&28&1&27, 46, 159, 433, 502\\ 21&36&9&2&27, 32\\ 85&136&17&2&6, 7, 8 
\end{tabularx}
\end{center}
\end{table}

\begin{table}[H]
\caption{Complex Homogeneous ETFs} \label{tbl:cETFs}
\begin{multicols}{2}
\begin{center}
\tiny
\begin{tabularx}{0.48\textwidth}{rrrrX}
\footnotesize{$d$} & \footnotesize{$n$} & \footnotesize{$m$} & \footnotesize{$t$} & \footnotesize{$k$} \\ \hline
2&4&8&1&12\\ 2&4&9&2&23\\ 2&4&16&1&59, 60, 438, 439, 726, 728, 729, 732, 1038, 1040, 1542, 1670\\ 3&7&7&1&1, 3\\ 3&7&8&2&25, 36\\ 3&7&14&1&1, 5, 6, 9, 11, 18, 21, 29, 35, 44\\ 3&7&21&1&6, 7, 11, 39, 50, 59, 60, 61, 79, 80, 86, 100\\ 3&7&28&1&11, 13, 14, 16, 17, 19, 20, 22, 27, 31, 37, 38, 39, 40, 44, 59, 60, 61, 62, 63, 64, 65, 66, 69, 79, 85, 89, 97, 99, 100, 101, 102, 103, 104, 110, 111, 112, 113, 114, 115, 116, 117, 118, 145, 150, 151, 154, 155, 156, 157, 160, 171, 173, 174, 177, 178, 179, 180, 181, 182, 187, 188, 189, 190, 191, 192, 216, 217, 220, 221, 222, 223, 226, 235, 236, 238, 255, 260, 262, 263, 264, 265, 266, 267, 268, 273, 286, 289, 290, 293, 294, 297, 298, 299, 300, 313, 314, 318, 319, 320, 321, 332, 333, 334, 335, 336, 337, 338, 339, 340, 341, 342, 343, 344, 345, 346, 365, 370, 372, 375, 376, 385, 389, 392, 398, 399, 400, 401, 402, 403, 404, 405, 406, 407, 408, 409, 410, 411, 412, 413, 436, 437, 444, 445, 447, 460, 461, 469, 470, 471, 472, 473, 474, 475, 476, 477, 478, 488, 503, 504, 507, 508, 512, 513, 514, 538, 541, 547, 549, 550, 551, 552, 553, 554, 555, 556, 557, 558, 559, 560, 561, 562, 563, 564, 565, 566, 567, 568, 569, 570, 571, 572, 588, 589, 590, 591, 631, 634, 635, 643, 648, 649, 656, 657\\ 3&9&27&1&6, 32, 50, 83, 212\\ 4&7&7&1&1, 3\\ 4&7&8&2&25, 36\\ 4&7&14&1&1, 5, 6, 9, 11, 18, 21, 29, 35, 44\\ 4&7&21&1&6, 7, 11, 39, 50, 59, 60, 61, 79, 80, 86, 100\\ 
4&7&28&1&11, 13, 14, 16, 17, 19, 20, 22, 27, 31, 37, 38, 39, 40, 44, 59, 60, 61, 62, 63, 64, 65, 66, 69, 79, 85, 89, 97, 99, 100, 101, 102, 103, 104, 110, 111, 112, 113, 114, 115, 116, 117, 118, 145, 150, 151, 154, 155, 156, 157, 160, 171, 173, 174, 177, 178, 179, 180, 181, 182, 187, 188, 189, 190, 191, 192, 216, 217, 220, 221, 222, 223, 226, 235, 236, 238, 255, 260, 262, 263, 264, 265, 266, 267, 268, 273, 286, 289, 290, 293, 294, 297, 298, 299, 300, 313, 314, 318, 319, 320, 321, 332, 333, 334, 335, 336, 337, 338, 339, 340, 341, 342, 343, 344, 345, 346, 365, 370, 372, 375, 376, 385, 389, 392, 398, 399, 400, 401, 402, 403, 404, 405, 406, 407, 408, 409, 410, 411, 412, 413, 436, 437, 444, 445, 447, 460, 461, 469, 470, 471, 472, 473, 474, 475, 476, 477, 478, 488, 503, 504, 507, 508, 512, 513, 514, 538, 541, 547, 
\end{tabularx}

\begin{tabularx}{0.48\textwidth}{rrrrX}
\footnotesize{$d$} & \footnotesize{$n$} & \footnotesize{$m$} & \footnotesize{$t$} & \footnotesize{$k$} \\ \hline
4&7&28&1&549, 550, 551, 552, 553, 554, 555, 556, 557, 558, 559, 560, 561, 562, 563, 564, 565, 566, 567, 568, 569, 570, 571, 572, 588, 589, 590, 591, 631, 634, 635, 643, 648, 649, 656, 657\\ 
4&8&16&1&715\\ 4&13&13&1&1, 3\\ 4&13&26&1&5, 64\\ 4&13&27&2&422\\ 5&10&30&1&45, 98\\ 5&11&11&1&1, 3\\ 5&11&22&1&5, 23, 28, 33\\ 5&21&21&1&6, 7, 11\\ 6&9&27&1&32, 50, 83, 212\\ 6&11&11&1&1, 3\\ 6&11&22&1&5, 23, 28, 33\\ 6&16&16&1&2, 4, 5, 15, 16, 17, 18, 19, 21, 26, 27, 28, 30, 34, 37, 57, 62, 63, 81, 115, 145, 146, 148, 176, 181, 184, 399, 428, 443, 759\\ 7&15&15&1&1, 3, 6\\ 7&15&16&2&447, 777, 1079\\ 7&15&30&1&5, 11, 17, 18, 20, 28, 50, 52, 53, 55, 64, 105, 108, 112, 113, 116, 118, 121, 202, 203, 204, 205, 208, 212, 218, 220, 323, 325, 326, 327, 331, 345, 350, 357, 490, 492, 493, 494, 513, 523, 526, 528, 677, 680, 682, 691, 693, 698, 701, 702, 883, 884, 886, 887, 891, 895, 900, 903, 1077, 1078, 1081, 1089, 1091, 1095, 1096, 1098, 1273, 1274, 1275, 1279, 1293, 1302, 1307\\ 7&28&28&1&27\\ 8&15&15&1&1, 3, 6\\ 8&15&16&2&447, 777, 1079\\ 8&15&30&1&5, 11, 17, 18, 20, 28, 50, 52, 53, 55, 64, 105, 108, 112, 113, 116, 118, 121, 202, 203, 204, 205, 208, 212, 218, 220, 323, 325, 326, 327, 331, 345, 350, 357, 490, 492, 493, 494, 513, 523, 526, 528, 677, 680, 682, 691, 693, 698, 701, 702, 883, 884, 886, 887, 891, 895, 900, 903, 1077, 1078, 1081, 1089, 1091, 1095, 1096, 1098, 1273, 1274, 1275, 1279, 1293, 1302, 1307\\ 9&13&13&1&1, 3\\ 9&13&26&1&5, 64\\ 9&13&27&2&422\\ 9&19&19&1&1, 3, 5\\ 10&16&16&1&2, 4, 5, 15, 16, 17, 18, 19, 21, 26, 27, 28, 30, 34, 37, 57, 62, 63, 81, 115, 145, 146, 148, 176, 181, 184, 399, 428, 443, 759\\ 10&19&19&1&1, 3, 5\\ 11&23&23&1&3\\ 12&23&23&1&3\\ 13&27&27&1&21, 134, 292\\ 14&27&27&1&21, 134, 292\\ 16&21&21&1&6, 7, 11\\ 21&28&28&1&27
\end{tabularx}
\end{center}
\end{multicols}
\end{table}

\section{A Gelfand pair involving the Heisenberg group}

We now turn our attention to the Heisenberg group, with the goal of constructing ETFs with Heisenberg symmetry. In this section, we identify a Gelfand pair having a Heisenberg group as a regular subgroup. We will leverage that Gelfand pair in the next section to construct a new infinite family of optimal line packings.

For the remainder of the paper, we fix a finite abelian group $A$ of odd order, with the group operation written additively. The exponent of $A$ is $\exp(A)$, the smallest positive integer $n$ such that $n\cdot a = 0_A$ for all $a \in A$. Recall that $\hat{A}$ is the Pontryagin dual group, consisting of all homomorphisms $\alpha \colon A \to \T$, under the operation of pointwise multiplication. We again use additive notation for $\hat{A}$, and also write $\langle a, \alpha \rangle$ for the image of $a\in A$ under $\alpha \in \hat{A}$. All such images necessarily lie in the cyclic subgroup $C_{\exp(A)}\subset \T$ of order $\exp(A)$, whose group operation we express multiplicatively. As a consequence of Lagrange's Theorem, $\exp(A)$ is odd, and squaring is an automorphism in $C_{\exp(A)}$; we express its inverse with the notation $z \mapsto z^{1/2}$. Similarly, $z^{-1/2}$ denotes the multiplicative inverse of $z^{1/2}$.

\begin{defn}
The \emph{Heisenberg group} over $A$ is the set $H = A \times \hat{A} \times C_{\exp(A)}$ under the operation
\[ (a_1, \alpha_1, z_1)\cdot (a_2,\alpha_2,z_2) = \left(a_1 + a_2, \alpha_1 + \alpha_2, z_1z_2 \langle a_2, \alpha_1 \rangle^{1/2} \langle a_1, \alpha_2 \rangle^{-1/2} \right). \]
\end{defn}
For a more compact notation, we let $K = A \times \hat{A}$ and introduce the \emph{symplectic form} ${[\cdot , \cdot ]}\colon K \times K \to C_{\exp(A)}$ given by
\[ [ (a_1,\alpha_1), (a_2, \alpha_2) ] = \langle a_2, \alpha_1 \rangle \langle a_1, \alpha_2 \rangle^{-1} \qquad (a_i \in A, \alpha_i \in \hat{A}). \]
Then $H = K \times C_{\exp(A)}$ as a set, with the group operation given by
\[ (u_1,z_1)\cdot (u_2,z_2) = (u_1 + u_2, z_1z_2[u_1,u_2]^{1/2}) \qquad (u_i \in K, z_i \in C_{\exp(A)}), \]
and inverses given by $(u,z)^{-1} = (-u,z^{-1})$ for $u\in K$ and $z\in C_{\exp(A)}$.

\begin{defn}
The \emph{symplectic group} over $A$ is the subgroup $\Sp(K) \subset \Aut(K)$ of all automorphisms that preserve the symplectic form. In other words,
\[ \Sp(K) = \{ \sigma \in \Aut(K) : [ \sigma( u_1 ), \sigma( u_2 ) ] = [u_1,u_2] \text{ for all }u_1,u_2 \in K \}. \]
\end{defn}
It acts on $H$ with automorphisms
\[ \sigma \cdot (u,z) = \left( \sigma( u ), z \right) \qquad (\sigma \in \Sp(K), u \in K, z \in C_{\exp(A)}), \]
as the reader can verify. 

Our goal in this section is to prove the following.

\begin{theorem} \label{thm:spgp}
$(H \rtimes \Sp(K), \Sp(K))$ is a Gelfand pair.
\end{theorem}

Our description of the symplectic group is essentially due to Weil \cite{Weil64}. In the special case where $A = \Z_m$, $\Sp(K)$ can be identified with $\SL(2,\Z_m)$, in which case the action on $H$ is closely associated with the Clifford group \cite{App05}. For a continuous version of Theorem~\ref{thm:spgp}, we refer the reader to \cite{BJR90}. Benson and Ratcliff proved a stronger version of Theorem~\ref{thm:spgp} in the special case where $A$ is the additive group of a finite field \cite{BR08}. They also have descriptions of the corresponding spherical functions in that case \cite{BR09}.

Our proof of Theorem~\ref{thm:spgp} requires a detailed understanding of the orbits of $\Sp(K)$ on $K$, which we now review. The following is due to Dutta and Prasad \cite{DP11,DP15}. By the Fundamental Theorem of Finitely Generated Abelian Groups, we can write $A = \prod_p A_p$, where the direct product is over all primes $p$ dividing $|A|$, and each $A_p$ has the form
\[ A_p \cong \Z_{p^{\lambda_{p,1}}}\times \dotsb \times \Z_{p^{\lambda_{p,l_p}}} \]
for a sequence $\lambda_p = (\lambda_{p,1}, \dotsc, \lambda_{p,l_p})$ of positive integers. Then $\hat{A} = \prod_p \hat{A}_p$ and $K = \prod_p K_p$ with 
\begin{equation} \label{eq:Kpd}
K_p = A_p \times \hat{A}_p \cong (\Z_{p^{\lambda_{p,1}}}\times \dotsb \times \Z_{p^{\lambda_{p,l_p}}}) \times ( \Z_{p^{\lambda_{p,1}}}\times \dotsb \times \Z_{p^{\lambda_{p,l_p}}} )
\end{equation}
for each $p$.

To each sequence $\lambda_p$, we associate the set
\[ P_{\lambda_p} = \{ (v,k) \in \Z^2 : k \in \{\lambda_{p,1},\dotsc,\lambda_{p,l_p}\} \text{ and }0 \leq v \leq k\}, \]
which is partially ordered by the relation
\[ (v_1,k_1) \geq (v_2,k_2) \iff v_2 \geq v_1 \text{ and } k_2 - v_2 \leq k_1 - v_1. \]
Recall that an \emph{order ideal} in $P_{\lambda_p}$ is a (possibly empty) set $I\subset P_{\lambda_p}$ with the property that ${(v_1,k_1)} \in I$ and ${(v_2,k_2)} \leq {(v_1,k_1)}$ imply ${(v_2,k_2)} \in I$.

\begin{theorem} \cite[Theorem~5.4]{DP11}. \label{thm:Aporb}
The order ideals of $P_{\lambda_p}$ are in one-to-one correspondence with the $\Aut(A_p)$-orbits of $A_p$.
\end{theorem}

\begin{theorem} \cite[Theorem~4.5]{DP15}. \label{thm:Sporb}
The order ideals of $P_{\lambda_p}$ are in one-to-one correspondence with the $\Sp(K_p)$-orbits of $K_p$.
\end{theorem}

Taken together, these yield the following corollary, implicitly contained in the proof of \cite[Theorem~4.5]{DP15}.

\begin{cor} \label{cor:sporb}
$\Sp(K)$-orbits in $K$ are identical to $\Aut(K)$-orbits.
\end{cor}

\begin{proof}
First, we prove the corollary for one of the subgroups $K_p \subset K$. Applying Theorem~\ref{thm:Aporb} with $K_p$ in place of $A_p$, and keeping in mind the decomposition \eqref{eq:Kpd}, we see that
\begin{equation} \label{eq:autsp}
| \{ \text{$\Aut(K_p)$-orbits in $K_p$} \} | = |\{ \text{order ideals in $P_{\lambda_p}$} \}| =  |\{ \text{$\Sp(K_p)$-orbits in $K_p$} \}|.
\end{equation}
Since $\Sp(K_p) \subset \Aut(K_p)$, every $\Sp(K_p)$-orbit is contained in a unique $\Aut(K_p)$-orbit, and every $\Aut(K_p)$-orbit is a disjoint union of $\Sp(K_p)$-orbits. Thus, containment produces a surjection $\{ \text{$\Sp(K_p)$-orbits} \} \to \{ \text{$\Aut(K_p)$-orbits} \}$, which must be an injection by \eqref{eq:autsp}. It follows that every $\Aut(K_p)$-orbit coincides with a single $\Sp(K_p)$-orbit.

Next, any choice of automorphisms $\sigma_p \in \Aut(K_p)$, $p$ prime and dividing $|A|$, gives an automorphism $\sigma \in \Aut(K)$, defined by
\[ \sigma\left( \sum_p u_p \right) = \sum_p \sigma_p(u_p) \qquad (u_p \in K_p). \]
We express this by saying that $\sigma = (\sigma_p)_p$. On the other hand, since the $K_p$'s have relatively prime orders, it is easy to show that any $\sigma \in \Aut(K)$ maps each $K_p$ back into itself, and therefore takes the form $\sigma = (\sigma_p)_p$ with $\sigma_p = \left. \sigma\right|_{K_p}$. Hence, we can identify 
\[ \Aut(K) \cong \prod_p \Aut(K_p). \]

We claim that $\Sp(K) \cong \prod_p \Sp(K_p)$ under the same identification. If $[\cdot, \cdot ]_p$ denotes the symplectic form on $K_p$, then for any choice of $a_p,b_p \in A_p$ and $\alpha_p,\beta_p \in \hat{A}_p$, we have
\begin{align*}
\left[ \sum_p (a_p,\alpha_p),\ \sum_p (b_p,\beta_p) \right] &= \left\langle \sum_p b_p, \sum_p \alpha_p \right\rangle \left\langle \sum_p a_p, \sum_p \beta_p \right\rangle^{-1} \\[5 pt]
&= \prod_p \langle b_p, \alpha_p \rangle \cdot \prod_p \langle a_p, \beta_p \rangle^{-1}  \\[5 pt]
&= \prod_p \left[ (a_p, \alpha_p), (b_p, \beta_p) \right]_p.
\end{align*}
Equivalently,
\[ \left[ \sum_p u_p, \sum_p v_p \right] = \prod_p [ u_p, v_p ]_p \qquad (u_p,v_p \in K_p). \]
From this it is clear that $\sigma \in \Sp(K)$ if and only if $\left. \sigma \right|_p \in \Sp(K_p)$ for every $p$. 

Now, if $u = \sum_p u_p \in K$, $u_p \in K_p$, and $\sigma = (\sigma_p)_p \in \Aut(K)$, then for each $p$ we can find $\sigma_p' \in \Sp(K_p)$ such that $\sigma_p'(u_p) = \sigma_p(u_p)$. By the above, $\sigma':=(\sigma_p')_p$ lies in $\Sp(K)$ and satisfies
\[ \sigma'(u) = \sum_p \sigma_p'(u_p) = \sum_p \sigma_p(u_p) = \sigma(u). \]
Thus, the orbit of $u \in K$ under $\Aut(K)$ is the same as its orbit under $\Sp(K)$.
\end{proof}

Now we can prove our main result.

\begin{proof}[Proof of Theorem~\ref{thm:spgp}]
By Theorem~\ref{thm:milk}, it suffices to prove that $L^2(H)^{\Sp(K)}$ is commutative. Given $(w,z) \in H$, we write $\delta_{(w,\, z)} \in L^2(H)$ for the corresponding point mass. If $f_1,f_2 \in L^2(H)^{\Sp(K)}$ are characteristic functions of $\Sp(K)$-orbits in $H$, then there are $\Sp(K)$-orbits $\mathcal{O}_1,\mathcal{O}_2 \subset K$ and $z_1,z_2 \in C_{\exp(A)}$ such that
\[ f_i = \sum_{u\in \mathcal{O}_i} \delta_{(u,\, z_i)} \qquad (i=1,2). \]
Consequently,
\begin{align*}
f_1 * f_2 &= \sum_{u \in \mathcal{O}_1} \sum_{v\in \mathcal{O}_2} \delta_{(u,\, z_1)\cdot (v,\, z_2)} \\[5 pt]
&= \sum_{u\in \mathcal{O}_1} \sum_{v\in \mathcal{O}_2} \delta_{\left(u+v,\, z_1z_2[u,v]^{1/2}\right)}.
\end{align*}
Given any $(w,z) \in H$, this means that
\begin{equation} \label{eq:f1f2}
(f_1 * f_2)(w,z) = \left|\{ (u,v) \in \mathcal{O}_1 \times \mathcal{O}_2 : u + v = w \text{ and } z_1z_2[u,v]^{1/2} = z \} \right|.
\end{equation}

Now let $\sigma \in \Aut(K)$ be given by $\sigma(a,\alpha) = (-a,\alpha)$ for $a\in A$ and $\alpha \in \hat{A}$. It satisfies the convenient identity
\[ [ \sigma(u), \sigma(v) ] = [v,u] \qquad (u,v \in K). \]
By Corollary~\ref{cor:sporb}, there is a $\sigma' \in \Sp(K)$ such that
\[ \sigma'\cdot (w,z) = \left( \sigma'(w),z \right) = \left( \sigma(w), z \right), \]
and by Theorem~\ref{thm:milk}, $f_1 * f_2$ is constant on $\Sp(K)$-orbits. Hence, 
\[ (f_1 * f_2)(w,z) = (f_1*f_2)\left(\sigma(w),z\right). \]
Another application of Corollary~\ref{cor:sporb} shows that $\sigma$ restricts to a bijection of each $\mathcal{O}_i$, $i=1,2$, onto itself. Thus, we can replace each of $u,v,w$ with $\sigma(u),\sigma(v),\sigma(w)$, respectively, in the appropriate place in \eqref{eq:f1f2} to obtain
\begin{align*}
(f_1 * f_2)(w,z) &= \left| \{ (u,v) \in \mathcal{O}_1 \times \mathcal{O}_2 : \sigma(u) + \sigma(v) = \sigma(w) \text{ and } z_1 z_2 [ \sigma(u), \sigma(v) ]^{1/2} = z \} \right| \\
&= \left| \{ (u,v) \in \mathcal{O}_1 \times \mathcal{O}_2 : u + v = w \text{ and } z_1 z_2 [ v, u ]^{1/2} = z \} \right|.
\end{align*}
Comparing with \eqref{eq:f1f2}, we see that $f_1 * f_2 = f_2 * f_1$. These were arbitrary elements of a basis for $L^2(H)^{\Sp(K)}$, so the latter is commutative, and $({H \rtimes \Sp(K)}, {\Sp(K)})$ is a Gelfand pair.
\end{proof}

\section{ETFs from Heisenberg groups} \label{sec:Heis}

As a consequence of Theorem~\ref{thm:spgp}, finitely many out of the uncountable class of Parseval $H$-frames are distinguished by also being $(H \rtimes \Sp(K), \Sp(K))$-frames. In this section, we show that the projective reduction of one of those frames is an ETF.

For the remainder of this section, we fix a character $\gamma \in \hat{C}_{\exp(A)}$, chosen in such a way that $\gamma \colon C_{\exp(A)} \to C_{\exp(A)}$ is an isomorphism. Given $\alpha \in \hat{A}$, we write $\gamma \alpha \in \hat{A}$ for the composition $\gamma \circ \alpha \colon A \to C_{\exp(A)}$. Then the mapping $\alpha \mapsto \gamma \alpha$ is an isomorphism $\hat{A} \cong \hat{A}$, with inverse $\alpha \mapsto \gamma^{-1} \alpha$. Finally, given $a\in A$, we write $\frac{1}{2} a \in A$ for the unique element with $\frac{1}{2}a + \frac{1}{2} a = a$; this is possible since $A$ has odd order. The mapping $a \mapsto \frac{1}{2} a$ is an isomorphism $A \cong A$, with inverse $a \mapsto 2a$.

\begin{defn}
The Schr\"odinger representation associated with $\gamma$ is the map $\pi_\gamma \colon H \to U(L^2(A))$ that for $(a,\alpha,z) \in H$ and $f\in L^2(A)$ is given by
\[ [\pi_\gamma(a,\alpha,z) f](b) = \gamma(z \langle b- \tfrac{1}{2} a , \alpha \rangle ) f(b-a) \qquad (b \in A). \]
\end{defn}

We leave it up to the reader to verify that $\pi_\gamma$ is a unitary representation of $H$. The choice of $\gamma$ determines $\pi_\gamma$ up to unitary equivalence, as shown by the following proposition. Here, we use the simplified notation $H = K \times C_{\exp(A)}$.

\begin{prop}
The trace character $\chi_\gamma$ of $\pi_\gamma$ is given by
\begin{equation} \label{eq:sc}
\chi_\gamma(u,z) = \begin{cases}
\gamma(z)|A|, & \text{ if }u=0_K; \\
0, & \text{ if }u\neq 0_K
\end{cases} \qquad (u\in K, z\in C_{\exp(A)}).
\end{equation}
Consequently, $\pi_\gamma$ is an irreducible representation.
\end{prop}

\begin{proof}
Given $b\in A$, we write $\delta_b \in L^2(A)$ for the corresponding point mass. Then $\{\delta_b\}_{b\in A}$ is an orthonormal basis, which we use to compute the trace of $\pi_\gamma(a,\alpha,z)$ for $(a,\alpha,z) \in H$:
\begin{align*}
\chi_\gamma(a,\alpha,z) &= \sum_{b\in A} \langle \pi_\gamma(a,\alpha,z) \delta_b, \delta_b \rangle \\[5 pt]
&= \sum_{b\in A} [ \pi_\gamma(a,\alpha,z) \delta_b ](b) \\[5 pt]
&= \sum_{b\in A} \gamma(z \langle b-\tfrac{1}{2} a, \alpha \rangle ) \delta_b(b-a).
\end{align*}
When $a \neq 0_A$, every term in the sum vanishes, so that $\chi_\gamma(a,\alpha,z) = 0$. On the other hand, when $a = 0_A$, we have
\[ \chi_\gamma(0_A, \alpha, z) = \gamma(z) \sum_{b\in A} \langle b, \gamma \alpha \rangle. \]
Here again, any nonzero choice of $\alpha$ makes $\gamma \alpha \in \hat{A}$ nontrivial, so that $\chi_\gamma(0_A,\alpha,z) = 0$. Thus, $\chi_\gamma(u,z) = 0$ whenever $u\in K$ is nonzero. On other hand, when $u=0_K$, we get
\[ \chi_\gamma(0_A, 0_{\hat{A}}, z) = \gamma(z)|A|. \]

To see that $\pi_\gamma$ is irreducible, we simply compute
\[ \frac{1}{|H|} \sum_{(u,z) \in H} | \chi_\gamma(u,z) |^2 = \frac{1}{|A|^2\cdot \exp(A)} \sum_{z\in C_{\exp(A)}} |\chi_\gamma(0,z)|^2 = \frac{1}{|A|^2\cdot \exp(A)} \sum_{z\in C_{\exp(A)}} |A|^2 = 1. \]
\end{proof}

We are going to construct a representation of $H$ on the space $\HS(L^2(A))$ of operators $T\colon L^2(A) \to L^2(A)$, with the Hilbert-Schmidt inner product $\langle S, T \rangle = \tr(ST^*)$.

\begin{prop}
An orthonormal basis for $\HS(L^2(A))$ is given by $\{ |A|^{-1/2} \cdot \pi_\gamma(u,1) \}_{u \in K}$.
\end{prop}

\begin{proof}
For $u,v\in K$, we compute
\begin{align*}
\langle \pi_\gamma(u,1), \pi_\gamma(v,1) \rangle &= \tr\left[ \pi_\gamma(u,1)\cdot \pi_\gamma(v,1)^* \right] \\
&= \tr\left[ \pi_\gamma(u,1)\cdot \pi_\gamma(-v,1) \right] \\
&= \tr\left[ \pi_\gamma(u-v,[u,-v]^{1/2}) \right].
\end{align*}
Using the identity $[u,-u] = 1$ and the trace character \eqref{eq:sc}, we see that
\[ \langle \pi_\gamma(u,1), \pi_\gamma(v,1) \rangle = \begin{cases}
|A|, & \text{if }u=v, \\
0, & \text{if }u\neq v.
\end{cases}. \]
Thus, $\{ |A|^{-1/2} \cdot \pi_\gamma(u,1)\}_{u \in K}$ are $|K| = |A|^2$ orthonormal vectors in a space of the same dimension, and therefore form an orthonormal basis.
\end{proof}

Closely related to $\pi_\gamma$ is the representation $\rho_\gamma \colon H \to U(\HS(L^2(A)))$ given by left operator multiplication:
\[ \rho_\gamma(u,z)(T) = \pi_\gamma(u,z)\cdot T \qquad \left( (u,z) \in H,\, T\in \HS(L^2(A)) \right). \]
Since $\pi_\gamma$ is unitary, $\rho_\gamma$ is, too. 

Unlike $\pi_\gamma$, $\rho_\gamma$ is a reducible representation. In particular, let
\[ V_E = \{ f\in L^2(A) : f(-a) = f(a) \text{ for all }a\in A\} \]
and
\[ V_O = \{ f\in L^2(A) : f(-a) = -f(a) \text{ for all }a\in A\} \]
be the spaces of even and odd functions in $L^2(A)$, respectively; orthogonal projections onto these are given by
\[ (P_E f)(a) := \frac{ f(a) + f(-a) }{2}\qquad (f \in L^2(A),\, a \in A) \]
and
\[ (P_O f)(a) := \frac{ f(a) - f(-a) }{2} \qquad (f \in L^2(A),\, a \in A). \]
We define
\[ \H_E = \{ T \in \HS(L^2(A)) : TP_E = T \} \]
and
\[ \H_O = \{ T \in \HS(L^2(A)) : TP_O = T \}. \]
These are easily seen to be invariant under $\rho_\gamma$. We can identify $\H_E \cong \HS(V_E,L^2(A))$ and $\H_O \cong \HS(V_O, L^2(A))$. Consequently,
\[ \dim \H_E = \frac{ |A|(|A|+1) }{2} \qquad \text{and} \qquad \dim \H_O = \frac{ |A| ( |A| - 1 ) }{2}. \]

\begin{theorem} \label{thm:HETF}
The projectively reduced $H$-frames
\[ \Phi_E:= \{ \rho_\gamma(u,1) P_E \}_{u\in K} \]
and
\[ \Phi_O := \{ \rho_\gamma(u,1) P_O \}_{u\in K} \]
are ETFs for $\H_E$ and $\H_O$, respectively. Explicitly, for $u,v \in K$ we have
\begin{equation} \label{eq:PEG}
\langle \rho_\gamma(u,1) P_E,\, \rho_\gamma(v,1) P_E \rangle = \begin{cases}
\frac{ |A| + 1 }{2}, & \text{ if } u=v; \\
\frac{1}{2} \gamma \left( [u,v]^{1/2} \right), & \text{ if } u \neq v;
\end{cases}
\end{equation}
and
\begin{equation} \label{eq:POG}
\langle \rho_\gamma(u,1) P_O,\, \rho_\gamma(v,1) P_O \rangle = \begin{cases}
\frac{ |A| - 1 }{2}, & \text{ if } u=v; \\
-\frac{1}{2} \gamma \left( [u,v]^{1/2} \right), & \text{ if } u \neq v.
\end{cases}
\end{equation}
\end{theorem}

\begin{proof}
We will prove the theorem for $\Phi_E$ first. Let $R \colon L^2(A) \to L^2(A)$ be the reversal operator given by $(Rf)(a) = f(-a)$ for $f\in L^2(A)$ and $a\in A$. It is easy to see that $R^2=I$, that $R=R^*$, and that $P_E = \frac{1}{2}(I+R)$. If $\Phi_E':=\{ \rho_\gamma(u,z) P_E \}_{(u,z) \in H}$ is the full orbit of $P_E$ under $\rho_\gamma$, then its function of positive type $\varphi_E \colon H \to \C$ satisfies, for any $(u,z) \in H$,
\begin{align}
\varphi_E(u,z) &= \langle P_E,\, \rho_\gamma(u,z) P_E \rangle \label{eq:phiE1} \\[5 pt]
&= \frac{1}{4} \left( \langle I,\, \rho_\gamma(u,z)I \rangle + \langle I,\, \rho_\gamma(u,z) R \rangle + \langle R,\, \rho_\gamma(u,z) I \rangle + \langle R,\, \rho_\gamma(u,z) R \rangle \right) \notag \\[5 pt]
&= \frac{1}{2} \left(  \tr \left[ \pi_\gamma\left(-u,z^{-1}\right)\right] + \tr \left[ R \pi_\gamma\left(-u,z^{-1}\right) \right ] \right). \notag
\end{align}
On the other hand, whenever $a\in A$, $\alpha \in \hat{A}$, and $z \in C_{\exp(A)}$,
\begin{align*}
\tr \left[ R \pi_\gamma(a,\alpha,z) \right] &= \sum_{b\in A} \langle R \pi_\gamma(a,\alpha,z) \delta_b, \delta_b \rangle \\[5 pt]
&= \sum_{b\in A} [ R \pi_\gamma(a,\alpha,z) \delta_b ](b) \\[5 pt]
&= \sum_{b\in A} \gamma \left( z \langle -b - \tfrac{1}{2} a, \alpha \rangle \right) \delta_b(-b-a).
\end{align*}
For any $b \in A$, we have $b = -b-a$ if and only if $b = -\frac{1}{2} a$. Thus,
\[ \tr \left[ R \pi_\gamma(a,\alpha,z) \right] = \gamma\left(z  \langle \tfrac{1}{2} a - \tfrac{1}{2} a, \alpha \rangle \right) = \gamma(z). \]
Substituting this and \eqref{eq:sc} into \eqref{eq:phiE1}, we obtain
\begin{equation} \label{eq:phiE2}
\varphi_E(u,z) = \begin{cases}
\frac{1}{2}\gamma(z^{-1}) \left( |A| + 1 \right), & \text{if }u=0 \\
\frac{1}{2} \gamma(z^{-1}), & \text{if }u \neq 0
\end{cases}
\qquad \left( (u,z) \in H \right).
\end{equation}

By \eqref{eq:PS}, the projective stabilizer of $P_E$ is $\{0\} \times C_{\exp(A)} \subset H$, so $\Phi_E$ is a projective reduction of $\Phi_E'$ by Proposition~\ref{prop:PS}. The formula \eqref{eq:PEG} for the inner products of $\Phi_E$ follows from \eqref{eq:phiE2}, since for any $u,v \in K$,
\begin{align*}
\langle \rho_\gamma(u,1) P_E,\, \rho_\gamma(v,1) P_E \rangle &= \langle P_E,\, \rho_\gamma[(-u,1)\cdot (v,1)] P_E \rangle \\
&= \langle P_E,\, \rho_\gamma(v-u,[-u,v]^{1/2}) P_E \rangle \\
&= \varphi_E(v-u,[u,v]^{-1/2}).
\end{align*}
In particular, $\Phi_E$ is equiangular.

Finally, if $T \in \H_E$ is arbitrary, then
\[ \sum_{u\in K} \left| \langle T,\, \rho_\gamma(u,1) P_E \rangle \right|^2 = \sum_{u\in K} \left| \tr \left[ T P_E \pi_\gamma(u,1)^* \right] \right|^2. \]
Since $TP_E = P_E$, we have
\[ \sum_{u\in K} \left| \langle T,\, \rho_\gamma(u,1) P_E \rangle \right|^2 = \sum_{u\in K} \left| \tr \left[ T \pi_\gamma(u,1)^* \right] \right|^2 = \sum_{u\in K} \left| \langle T,\, \pi_\gamma(u,1) \rangle \right|^2, \]
and since $\{ |A|^{-1/2} \pi_\gamma(u,1) \}_{u\in K}$ is an orthonormal basis,
\[ \sum_{u\in K} \left| \langle T,\, \rho_\gamma(u,1) P_E \rangle \right|^2 = |A|\cdot \Norm{T}^2. \]
Thus, $\Phi_E$ is an equiangular tight frame with frame constant $|A|$.

The proof for $\Phi_O$ proceeds similarly, using the identity $P_O = \frac{1}{2}(I-R)$. In this case, the $H$-frame $\Phi_O':= \{ \rho_\gamma(u,z) P_O \}_{(u,z) \in H}$ has function of positive type $\varphi_O \colon H \to C$ given by
\begin{equation} \label{eq:phiO}
\varphi_O(u,z) = \begin{cases}
\frac{1}{2}\gamma(z^{-1}) \left( |A| - 1 \right), & \text{if }u=0 \\
-\frac{1}{2} \gamma(z^{-1}), & \text{if }u \neq 0
\end{cases}
\qquad \left( (u,z) \in H \right).
\end{equation}
Just as above, $\Phi_O$ is a projective reduction of $\Phi_O'$.
\end{proof}

\begin{rem}
In the proof above, the functions of positive type $\varphi_E$ and $\varphi_O$ are invariant under the action of $\Sp(K)$ on $H$, in the sense that
\[ \varphi_E(\sigma(u), z) = \varphi_E(u,z) \qquad \text{and} \qquad \varphi_O(\sigma(u), z) = \varphi_O(u,z) \]
for all $(u,z) \in H$ and $\sigma \in \Sp(K)$. By Corollary~\ref{cor:milk}, $\Phi_E$ and $\Phi_O$ are projective reductions of $({H\rtimes \Sp(K)}, \linebreak[0] {\Sp(K)})$-frames. There are only finitely many of these up to rescaling, as a consequence of Theorem~\ref{thm:spgp}. That finiteness was instrumental to the authors' discovery of Theorem~\ref{thm:HETF}, in that it made it possible to compute all tight $({H\rtimes \Sp(K)}, \linebreak[0] {\Sp(K)})$-frames for a small instance of $H$, and simply observe that one of these was an ETF.
\end{rem}

\section*{Acknowledgments}

Part of this research was conducted during the SOFT 2017: Summer of Frame Theory workshop at the Air Force Institute of Technology. Financial support was provided in part by NSF DMS 1321779, ARO W911NF-16-1-0008, AFOSR F4FGA06060J007, AFOSR Young Investigator Research Program award F4FGA06088J001, and by the Air Force Summer Faculty Fellowship Program. The authors thank Matt Fickus for ongoing, stimulating conversations.
The views expressed in this article are those of the authors and do not reflect the official policy or position of the United States Air Force, Army, Department of Defense, or the U.S.\ Government.


\bibliographystyle{abbrv}
\bibliography{gelfand_pairs}

\end{document}